\documentclass[a4paper,reqno]{amsart}

\usepackage[margin=2.5cm,papersize={19.75cm,28cm}]{geometry}
\usepackage[all]{xy}           
\usepackage{amssymb}           
\usepackage{hyperref}
\usepackage{eucal}
\usepackage{graphicx}
\usepackage{wrapfig}
\usepackage{tikz}
\usepackage{url}
\usepackage{epsfig}
\usepackage{graphicx}
\usepackage{bm}
\DeclareMathAlphabet{\mathpzc}{OT1}{pzc}{m}{it}

\newcommand{\R}{\mathbb{R}}      
\newcommand{\F}{\mathbb{F}}

\newcommand{\Ker}{\operatorname{Ker}}

\newcommand*{\llangle}{\langle\kern-2\nulldelimiterspace\langle}
\newcommand*{\rrangle}{\rangle\kern-2\nulldelimiterspace\rangle}
\newcommand*{\Bigrrangle}{\Big\rangle\kern-3\nulldelimiterspace\Big\rangle}
\newcommand*{\Bigllangle}{\Big\langle\kern-3\nulldelimiterspace\Big\langle}

\theoremstyle{plain} 
\newtheorem{definition}{Definition}[section]
\newtheorem{lemma}[definition]{Lemma}

\newtheorem{corollary}[definition]{Corollary}
\newtheorem{theorem}[definition]{Theorem}
\newtheorem{vpx}{Variational Principle}
\newenvironment{vp}
{\begin{vpx} \pushQED{\qed}}{\popQED \end{vpx}}

\theoremstyle{definition}
\newtheorem{example}[definition]{Example}
\newtheorem{remarkth}[definition]{Remark}
\newenvironment{remark}{\begin{remarkth}}{\hfill$\Diamond$\end{remarkth}}

\allowdisplaybreaks 

\begin{document}
\title[Geometric integrators for higher-order variational systems]{Geometric integrators for higher-order variational systems and their
application to optimal control}

\author[L. Colombo]{Leonardo Colombo}
\address{Leonardo Colombo: Department of Mathematics, University of Michigan, Ann Arbor, MI 48109, USA} \email{ljcolomb@umich.edu}

\author[S. Ferraro]{Sebasti\'an Ferraro}
\address{Sebasti\'an Ferraro: Departamento de Matem\'atica, Universidad Nacional del Sur, Instituto de Matem\'atica Bah\'ia Blanca, and CONICET, Av. Alem 1253, 8000 Bah\'ia Blanca, Argentina} \email{sferraro@uns.edu.ar}
\author[D.\ Mart\'\i n de Diego]{David Mart\'\i n de Diego}
\address{David Mart\'\i n de Diego:
Instituto de Ciencias Matem\'aticas (CSIC-UAM-UC3M-UCM), Campus de
Cantoblanco, UAM C/ Nicolas Cabrera, 15 - 28049 Madrid, Spain}
\email{david.martin@icmat.es}





\thanks{This work has been supported by MICINN (Spain) Grant MTM 2013-42870-P, ICMAT Severo Ochoa Project SEV-2011-0087 and
IRSES-project "Geomech-246981''. The research of S. Ferraro  has
been supported by CONICET Argentina (PIP 2010-2012 GI
11220090101018), ANPCyT Argentina (PICT 2010-2746) and SGCyT UNS}

\maketitle

\begin{abstract}
Numerical methods that preserve geometric invariants of the system, such as energy, momentum
or the symplectic form, are called geometric integrators. In this
paper we present a method to construct symplectic-momentum
integrators for higher-order Lagrangian systems. Given a regular
higher-order Lagrangian $L\colon T^{(k)}Q\to\R$ with $k\geq 1$, the
resulting discrete equations define a generally implicit numerical
integrator algorithm on $T^{(k-1)}Q\times T^{(k-1)}Q$ that
approximates the flow of the higher-order Euler--Lagrange equations
for $L$. The algorithm equations are called higher-order discrete
Euler--Lagrange equations and constitute a variational integrator for
higher-order mechanical systems. The general idea for those
variational integrators is to directly discretize Hamilton's
principle rather than the equations of motion in a way that
preserves the invariants of the original system, notably the symplectic
form and, via a discrete version of Noether's theorem, the momentum
map.

We construct an exact discrete Lagrangian $L_d^e$ using
the locally unique solution of the higher-order Euler--Lagrange equations for
$L$ with boundary conditions. By taking the discrete Lagrangian as an approximation of $L_d^e$, we obtain variational integrators for higher-order mechanical systems. We apply our techniques to optimal control
problems since, given a cost function, the optimal control problem is
understood as a second-order variational problem.
\end{abstract}

\tableofcontents

\section{Introduction}
This paper is concerned with the design of
geometric integrators for higher-order variational systems. 
The
study of higher-order variational systems has regularly attracted a lot
of attention from the applied and theoretical points of view (see \cite{LR1} and references therein). But
recently  there is a renewed interest in these systems due to new and
relevant applications in optimal control for robotics or
aeronautics, or the study of air traffic control and  computational
anatomy (\cite{
MR3111521,
MR1333770,
MR2864799,
MR2972054,
MR2861780,
HusseinBloch,
MR2580471,
MR1036158}).

A continuous higher-order system is modeled by a Lagrangian
on a higher-order tangent bundle $T^{(k)}Q$, that is, a function $L\colon
T^{(k)}Q \to {\mathbb R}$. The corresponding
Euler--Lagrange equations are a system of implicit $2k$-order
differential equations. Of course the explicit integration of most
of these Lagrangian systems is too complicated to integrate directly
or even it  is generically not possible.  In these cases, it is
necessary to discretize  the equations taking approximations  at
several points in time over the interval of integration.

Among the different numerical integrators that
one can derive for continuous higher-order systems, one of the most
successful ideas is to discretize first the variational principle
(instead of the equations of motion) and  to derive the numerical method applying discrete calculus of variations \cite{mawest,MR947601, welant}. The advantage of
this procedure is that automatically we have preservation of  some
of the geometric structures involved, like symplectic forms or
preservation of momentum, moreover, a good behavior of the
associated energy. These methods have their roots in the optimal
control literature in the 1960s \cite{MR0179019}.

In previous approaches (see for example \cite{belema, CoMdDZu2011, MR3136650}), the theory of discrete variational mechanics
for higher-order systems was derived using a discrete Lagrangian
$L_d\colon  Q^{k+1}\to \mathbb{R}$ where $Q^{k+1}$ is the
cartesian product of $k+1$ copies of the configuration manifold $Q$. There, $k+1$ points are used to approximate the positions and the higher-order
velocities (such as the standard  velocities,
 accelerations, jerks...) and to represent in this way elements of the higher-order tangent bundle $T^{(k)}Q$.

We will see in this paper that the most natural approach
is to take a discrete Lagrangian $L_d\colon  T^{(k-1)}Q\times T^{(k-1)}Q
\to{\mathbb R}$ since actually the discrete variational calculus is not based on the discretization of the Lagrangian itself, but on the discretization of the associated action.
 We will see that a suitable approximation of the action
 \begin{equation*}
 \int^h_0 L(q, \dot{q}, \ldots, q^{(k)})\; dt
 \end{equation*}
 is given by a Lagrangian of the form $L_d\colon  T^{(k-1)}Q\times T^{(k-1)}Q \to {\mathbb R}$. Moreover, we will derive a  particular choice
  of discrete Lagrangian which gives an exact
correspondence between discrete and continuous systems, the exact discrete Lagrangian.
For instance, if we take the Lagrangian $L(q, \dot{q}, \ddot{q})=\frac{1}{2}\ddot{q}^2$, the corresponding exact  discrete Lagrangian $L_{d}^{e}\colon TQ\times TQ\to\R$ is
\begin{align*}
L^e_d(q_0, v_0, q_h, v_h)
&=\int_0^hL(q(t), \dot{q}(t), \ddot{q}(t))\; dt\\
 &=\frac{6}{h^3}(q_0-q_h)^2+\frac{6}{h^2}(q_0-q_h)(v_0+v_h)+\frac{2}{h}(v_0^2+v_0v_h+v_h^2)
\end{align*}
where $q(t)$ is the unique solution of the Euler--Lagrange equations for $L$ verifying $q(0)=q_0$, $\dot{q}(0)=v_0$, $q(h)=q_h$, $\dot{q}(h)=v_h$ for $h$ small enough (see Section
\ref{existence}).

Observe from the previous example that now this  theory of
variational integrators for higher-order systems is even simpler,
since it fits directly into the standard discrete mechanics theory
for a discrete Lagrangian of the form $L_d\colon  M\times M\to
{\mathbb R}$ where $M=T^{(k-1)}Q$. We will show that if the original
Lagrangian is regular then so is the exact discrete
Lagrangian, in the sense of \cite{mawest}. Moreover, in the
corresponding applications, for instance in optimal control theory
or splines theory, typically we are dealing with initial and final
boundary conditions which are not necessary discretized, in contrast to
previously proposed methods \cite{MR2534806,MR2448687,MR2954746}.

The paper is structured as follows. In Section 2, we show that a regular higher-order Lagrangian system has a unique solution for given nearby endpoint conditions using a direct variational proof of existence and uniqueness of the local boundary value problem, which employs a regularization procedure. In Section 3 we introduce the notion of exact discrete Lagrangian for higher-order systems and we design the construction of variational integrators for higher-order Lagrangian systems taking approximations of the exact discrete Lagrangian. We obtain the discrete Euler--Lagrange equations for a discrete Lagrangian defined in the cartesian product of two copies of $T^{(k-1)}Q$. Section 4 is devoted to the study of the relation between the discrete and continuous dynamics. We show the relation between the discrete Legendre transformations and the continuous one and we also show that the exact discrete Lagrangian associated with a higher-order regular Lagrangian is also regular. 
Finally, in Section 5, we apply our techniques to study optimal control problems for fully actuated mechanical systems.

\section{Existence and uniqueness of solutions for the boundary value problem} \label{existence}

\subsection{Higher-order tangent bundles}
First we recall some basic facts about the higher-order
tangent bundle theory. For more details see \cite{CSC} and
\cite{LR1}.

Let $Q$ be a  differentiable manifold. We introduce the following equivalence relation in the set
$C^{k}(I, Q)$ of $k$-differentiable
curves from the interval $I\subseteq \R$ to $Q$, where $0\in I$. 
By definition, two curves $\gamma_1$ and $\gamma_2$ belonging to $C^{k}(I, Q)$
have contact of order  $k$ at $q_0 = \gamma_1(0) = \gamma_2(0)$ if
there is a local chart $(\varphi, U)$ of $Q$ such that $q_0 \in U$
and
\[\frac{d^s}{dt^s}\left(\varphi \circ \gamma_1(t)\right){\Big{|}}_{t=0} =
\frac{d^s}{dt^s} \left(\varphi
\circ\gamma_2(t)\right){\Big{|}}_{t=0}\; ,\]
for all $s = 0,\dots,k$. The equivalence class of a  curve $\gamma$ will be denoted by
$[\gamma ]_0^{(k)}$. The set of equivalence classes will be denoted
by $T^{(k)}Q$
and it is not hard to show that it has a natural structure of
differentiable manifold. Moreover, $ \tau_Q^k \colon T^{(k)} Q
\to Q$ where $\tau_Q^k \big([\gamma]_0^{(k)}\big) =
\gamma(0)$ is a fiber bundle called the \textit{tangent bundle of
order $k$} of $Q$. Clearly, $T^{(1)} Q = TQ$.

From a local chart $q^{(0)}=(q^i)$ on a neighborhood $U$ of $Q$ with
$i=1,\ldots,n=\dim Q$, it is possible to induce local coordinates
 $(q^{(0)},q^{(1)},\dots,q^{(k)})$ on
$T^{(k)}U=(\tau_Q^k)^{-1}(U)\equiv U\times (\R^{n})^{k}$. Sometimes we
will resort to the usual notation $q^{(0)}\equiv (q^{i})$,
$q^{(1)}\equiv (\dot{q}^i)$ and $q^{(2)}\equiv (\ddot{q}^i)$.

There is a canonical embedding $j_{k}\colon T^{(k)}Q\to T
T^{(k-1)}Q$ defined as
$j_k([\gamma]_0^{(k)})=[{\gamma}^{(k-1)}]_0^{(1)}$, where
${\gamma}^{(k-1)}$ is the lift of the curve $\gamma$ to
$T^{(k-1)}Q$; that is,  the curve ${\gamma}^{(k-1)}\colon I\to
T^{(k-1)}Q$ is given by $\gamma^{(k-1)}(t)=[\gamma_t]_0^{(k-1)}$
where $\gamma_t(s)=\gamma(t+s)$. In local coordinates,
\[
j_k(q^{(0)},q^{(1)},q^{(2)},...q^{(k)})=(q^{(0)},q^{(1)},\dots,q^{(k-1)};q^{(1)},
q^{(2)},\dots,q^{(k)})\; .
\]

\subsection{Hamilton's principle and considerations about the existence and uniqueness of solutions}
Let $L\colon T^{(k)}Q \to \mathbb{R}$ be a Lagrangian of order $k\geq 1$, of class $C^{k+1}$. Since our result will be local, we
assume from now on that $Q$ is an open subset of $\mathbb{R}^n$. Take coordinates
$(q^{(0)}, q^{(1)}, \dots, q^{(k)})$ on $T^{(k)}Q
\equiv Q\times (\mathbb{R}^n)^k$ as before. We suppose that $L$ is regular in
the sense that the Hessian matrix
\[
\left(\frac{\partial^2 L}{\partial q^{(k)i}\partial q^{(k)j}}\right)
\]
is a regular matrix. Let also $h>0$ be given. We can formulate Hamilton's principle as follows.

\begin{vp}\label{originalvp}
Find a $C^k$ curve $q\colon [0,h] \to Q$ such that it is a
critical point of the action
\[
S_h=\int_0^h L\left( q (t), \dot q (t), \dots, q^{(k)} (t)\right) dt
\]
among those curves whose first
$k-1$ derivatives are fixed at the endpoints, that is, with given values for
$q(0), \dot q(0),\,\dots,\, q^{(k-1)}(0)$ and $q(h), \dot q(h),\,\dots,\, q^{(k-1)}(h)$.
\end{vp}

Hamilton's principle is a constrained problem in the Banach space $C^k([0,h], \mathbb{R}^n)$. Now if $q(t)$ is a solution to this problem that is not only $C^k$ but $C^{2k}$, then it satisfies the well-known $k$\textsuperscript{th}-order Euler--Lagrange equations\footnote{For $k=1$, recall writing $\delta \dot q=\dot{(\delta q)}$ when deriving the Euler--Lagrange equations, assuming that $q$ is $C^2$.}
\begin{equation}\label{hoeulerlagrangeequationsvp}
\sum_{j=0}^k(-1)^j \frac{d^j}{dt^j}\frac{\partial L}{\partial q^{(j)}}=0.
\end{equation}
For a regular Lagrangian, \eqref{hoeulerlagrangeequationsvp} can be written as an explicit $2k$-order ordinary differential equation. Existence and uniqueness of solutions for the \emph{initial} value problem can be guaranteed using basic ODE theory. Doing the same for for the boundary value problem of finding a solution $q(t)$ of \eqref{hoeulerlagrangeequationsvp} with given values for $q(0), \dot q(0),\,\dots,\, q^{(k-1)}(0)$ and $q(h), \dot q(h),\,\dots,\, q^{(k-1)}(h)$ requires different techniques. For instance, in \cite[ch.\ 9]{higherorderbook} it is shown that there exists a unique solution to an explicit $2k$-order ODE with this kind of boundary conditions, for small enough $h$ and close enough boundary values.

In principle, however, there could exist solutions to Hamilton's variational principle that are $C^k$ but not $C^{2k}$, and thus do not satisfy \eqref{hoeulerlagrangeequationsvp}. Therefore, uniqueness of solutions to the variational principle cannot yet be guaranteed. One possibility for avoiding this situation is stating Hamilton's principle in the (smaller) $C^{2k}$ context from the beginning. In this section we proceed differently, acknowledging the fact the variational principle makes sense in the $C^k$ setting. We prove local existence and uniqueness of $C^k$ solutions to Hamilton's principle from a direct variational point of view. We will see that these solutions turn out to be automatically $C^{2k}$, so they satisfy Euler--Lagrange equations \emph{a posteriori}.

Our argument for the existence and uniqueness of solutions will involve a regularization procedure which follows closely the proof by
Patrick \cite{patrick} for first-order Lagrangians; the formulas, of
course, reduce to those in \cite{patrick} for order 1, but we
introduce an additional modification using orthonormal polynomials.

\subsection{Non-regularity of Hamilton's principle}
We want to determine whether there exists a unique solution curve
to Hamilton's principle, given endpoint conditions that are
close enough. The main obstacle for a straightforward affirmative
answer is that the local boundary value problem as stated above is
nonregular at $h=0$. That is, the constraint function
$g\colon C^{k}([0,h], Q)
\to (\R^{n})^{k}\times(\R^{n})^{k}$
\[
g\colon q( \cdot ) \mapsto \left(q (0), \dot q (0), \dots, q^{(k-1)} (0);q (h), \dot q (h), \dots, q^{(k-1)} (h)\right)
\]
maps into the diagonal of $T^{(k-1)}Q \times T^{(k-1)}Q$ for $h=0$
and is not therefore a submersion. For $h\neq 0$, the constraint
function is a submersion.

The approach consists in replacing this problem by an equivalent one
that is regular at $h=0$, and show that locally there is a unique
solution to the regularized problem.

\subsection{Regularization}
First we replace the space of curves on $Q$  in the variational
problem by the space of curves on $T^{(k)}Q$, and include additional
constraints. Denote an arbitrary curve by
\[
\left(q (t)=q^{[0]} (t), q^{[1]} (t), \dots, q^{[k]} (t)\right) \in T^{(k)}Q \equiv Q\times (\mathbb{R}^n)^k,
\]
$t\in [0,h]$. Here we have modified our notation for coordinates on $T^{(k)}Q$, using superscripts in square brackets to
make a distinction with the actual derivatives of $q(t)$.

\begin{vp}
Find a curve $(q^{[0]}(t), q^{[1]} (t), \dots, q^{[k]} (t))$ on
$T^{(k)}Q$, with $q^{[l]}\in C^{k-l}([0,h], \allowbreak \mathbb{R}^n)$, $l=0,
\dots, k$, such that it is a critical point of
\[
S_h=\int_0^hL\left(q^{[0]}(t), q^{[1]} (t), \dots, q^{[k]} (t)\right)dt
\]
subject to the constraints
\[
q^{[j+1]}(t)= \frac{dq^{[j]}}{dt}(t),\quad q^{[j]}(0)=q^{[j]}_1,\quad q^{[j]}(h)=q^{[j]}_2,\quad j=0, \dots, k-1,
\]
where $(q^{[0]}_i, q^{[1]}_i, \dots, q^{[k-1]}_i)$, $i=1,2$, are
given points in $T^{(k-1)}Q$.
\end{vp}

Now reparameterize the curve by defining
\[
Q^{[j]}(u)=q^{[j]}(hu),\quad j=0,\dots,k, \quad u\in[0,1].
\]
For $h>0$, the curve $(Q^{[0]}(u), \dots, Q^{[k]} (u))$ satisfies an
equivalent variational problem as follows. Since $h$ is a constant
for each instance of the problem, we can use
\[
\frac{1}{h} \int_0^hL\left(q^{[0]}(t), q^{[1]} (t), \dots, q^{[k]} (t)\right)dt=\int_0^1L\left(Q^{[0]}(u), \dots, Q^{[k]} (u)\right)du
\]
as an objective function. The first set of constraints becomes
\[
0= \frac{dq^{[j]}}{dt}(t)-q^{[j+1]}(t)= \left(\frac{1}{h} \frac{dQ^{[j]}}{du}(u)-Q^{[j+1]}(u) \right)_{u=t/h}
\]
where $j=0, \dots, k-1$.

The reparametrized variational principle is the following.

\begin{vp}\label{reparam_vp}
Find a curve $(Q^{[0]}(u), \dots, Q^{[k]} (u))$ on $T^{(k)}Q$,
$Q^{[l]}\in C^{k-l}([0,1], \mathbb{R}^n)$, $l=0, \dots, k$, that is
a critical point of
\begin{equation*}
S=\int_0^1L\left(Q^{[0]}(u), \dots, Q^{[k]} (u)\right)du,
\end{equation*}
subject to the constraints
\begin{align}
\frac{dQ^{[j]}}{du}(u)&=hQ^{[j+1]}(u),\label{k-1_order}\\
Q^{[j]}(0)&=q^{[j]}_1,\\
Q^{[j]}(1)&=q^{[j]}_2,
\end{align}
where $j=0, \dots, k-1$, and $(q^{[0]}_i, q^{[1]}_i, \dots,
q^{[k-1]}_i)$, $i=1,2$, are given points in $T^{(k-1)}Q$.
\end{vp}
The objective $S$ does not depend on $h$, and the constraints are
smooth through $h=0$.

\begin{remark}\label{vertical_curve_h0}
For $h=0$, the constraints \eqref{k-1_order} imply that
$Q^{[0]}(u)$, \dots, $Q^{[k-1]} (u)$ remain constant, which
restricts the possible values of the endpoint conditions in order to
have a compatible set of constraints. More precisely,
$q^{[j]}_1=q^{[j]}_2$ for $j=0,\ldots,k-1$; otherwise there would be
no curves satisfying the constraints. This kind of restriction also
appears in the original variational principle \ref{originalvp}.
Moreover, the problem becomes the unconstrained problem of finding a
curve $Q^{[k]}(u)\in C^0([0,1],\R^{n})$ such that it is a critical
point of
\[\int_{0}^{1}L\left(q^{[0]},\ldots,q^{[k-1]},Q^{[k]}(u)\right)du.\]
This means
\begin{equation*}
\frac{\partial L}{\partial q^{[k]}}\left(q^{[0]},q^{[1]},\ldots,q^{[k-1]},Q^{[k]}(u)\right)=0.
\end{equation*}

Differentiating with respect to $u$, and using the fact that the
Lagrangian is regular, we obtain that $Q^{[k]}(u)$ is constant.
\end{remark}

In preparation for the next step for regularization, let us solve
the constraints \eqref{k-1_order} to get
\[
Q^{[j]}(u)=Q^{[j]}(0)+h\int_0^u Q^{[j+1]}(s)\,ds,\quad j=0, \dots, k-1.
\]
This means that the functions $Q^{[j]}(u)$, $j=0, \dots, k-1$, can
be expressed in terms of $Q^{[j]}(0)$, \dots, $Q^{[k-1]}(0)$, the
function $Q^{[k]}(u)$ and $h$. For example, for $k=2$ we have
\begin{align*}
Q^{[1]}(u)&=Q^{[1]}(0)+h\int_0^u Q^{[2]}(s)\,ds,\\
Q^{[0]}(u)&=Q^{[0]}(0)+h\int_0^u Q^{[1]}(s)\,ds\\
&=Q^{[0]}(0)+huQ^{[1]}(0)+h^2\int_0^u\int_0^s Q^{[2]}(\tau)\,d\tau\,ds\\
&=Q^{[0]}(0)+huQ^{[1]}(0)+h^2\int_0^u(u-\tau) Q^{[2]}(\tau)\,d\tau.\\
\end{align*}

For a general $k$, and for $j=0, \dots, k-1$, an iterated change of
order of integration yields
\begin{equation}\label{Qj_in_terms_of_Qk}
Q^{[j]}(u)=Q^{[j]}(0)+\sum_{i=1}^{k-j-1}\frac{h^iu^i}{i!}Q^{[j+i]}(0)+h^{k-j}\int_0^u \frac{(u-s)^{k-j-1}}{(k-j-1)!}Q^{[k]}(s)\,ds.
\end{equation}
If the upper bound of summation is less than the lower bound, the sum is understood to be 0.

Note that taking $u=1$, the final endpoint data $(q^{[0]}_2, \dots, q^{[k-1]}_2)$
can now be written as
\begin{equation}\label{q2_as_poly_integral}
q^{[j]}_2=Q^{[j]}(1)=q^{[j]}_1+\sum_{i=1}^{k-j-1}\frac{h^i}{i!}q^{[j+i]}_1+h^{k-j}\int_0^1 \frac{(1-s)^{k-j-1}}{(k-j-1)!}Q^{[k]}(s)\,ds,
\end{equation}
so we define
\begin{equation}\label{def_of_z}
  z^{[j]}=\int_0^1 \frac{(1-s)^{k-j-1}}{(k-j-1)!}Q^{[k]}(s)\,ds=\frac{1}{h^{k-j}}\left(q^{[j]}_2-
\sum_{i=0}^{k-j-1}\frac{h^i}{i!}q^{[j+i]}_1\right).
\end{equation}
We will discuss the case $h=0$ in Remark \ref{blowup}.

Now replace the curves and endpoint data by just $Q^{[k]}(u)$,
$(q^{[0]}_1, \dots, q^{[k-1]}_1)$, and $(z^{[0]}, \dots,
z^{[k-1]})$, to get a new variational principle.

\begin{vp}\label{vp_with_z}
Given $h$, $(q^{[0]}_1, \dots, q^{[k-1]}_1)$ and $(z^{[0]}, \dots,
z^{[k-1]})$, find a continuous curve $Q^{[k]}\colon [0,1] \to
\mathbb{R}^n$ that is a critical point of
\[
\mathcal{S}=\int_0^1L\left(Q^{[0]}(u), \dots, Q^{[k]} (u)\right)du,
\]
where $Q^{[0]}(u)$, \dots, $Q^{[k-1]}(u)$ are defined as in
\eqref{Qj_in_terms_of_Qk} by
\begin{equation*}
Q^{[j]}(u)=q^{[j]}_1+\sum_{i=1}^{k-j-1}\frac{h^iu^i}{i!}q^{[j+i]}_1+h^{k-j}\int_0^u \frac{(u-s)^{k-j-1}}{(k-j-1)!}Q^{[k]}(s)\,ds,\quad j=0,\ldots, k-1
\end{equation*}
subject to the constraints
\begin{equation*}
  \int_0^1 \frac{(1-s)^{k-j-1}}{(k-j-1)!}Q^{[k]}(s)\,ds=z^{[j]},\quad j=0, \dots, k-1.\qedhere
\end{equation*}
\end{vp}

Observe that the constraint functions do not depend on $h$ and are linear on the curve $Q^{[k]}$.
This variational principle is already regular through $h=0$, as we will see when we proceed to find the solutions later.

\begin{remark}\label{blowup}
The data $q^{[0]}_1$, \dots, $q^{[k-1]}_1$, $z^{[0]}$, \dots,
$z^{[k-1]}$ can be transformed into the endpoint conditions for the
variational principle \ref{reparam_vp} in a straightforward way, for
any $h$, using \eqref{q2_as_poly_integral} and \eqref{def_of_z}. The
converse \eqref{def_of_z} is possible only for $h\neq 0$, in
principle. However, if $h=0$ let $(Q^{[0]}(u),\ldots,Q^{[k]}(u))$ a
solution for the variational principle \ref{reparam_vp} with
boundary conditions $(q_{1}^{[0]},\ldots, q^{[k-1]}_1)$ and
$(q^{[0]}_2,\ldots,\allowbreak q^{[k-1]}_2)$. Define $z^{[j]}$ by the constraint
in \eqref{vp_with_z}. Since $Q^{[k]}$ is constant and
$\frac{(1-s)^{k-j-1}}{(k-j-1)!}>0$ in $(0,1)$, to different values
of $Q^{[k]}$ correspond different values of $z^{[j]}$. Then
$Q^{[k]}$ is a solution of \ref{vp_with_z} with boundary
conditions $q^{[0]}_1$, \dots, $q^{[k-1]}_1$, $z^{[0]}$, \dots,
$z^{[k-1]}$.
\end{remark}

Finally, we will introduce a modification that will enable us to
carry out the computations in the next section easily. Consider the
inner product on $C^0([0,1],\mathbb{R})$ given by
\[
\langle f, g \rangle=
\int_0^1 f(s) g(s)\,ds.
\]

If $f\in C^0([0,1],\mathbb{R})$ and $V=(V_1, \dots, V_n)\in
C^0([0,1],\mathbb{R}^n)$ we define the bilinear operation
\begin{equation*}
\langle f, V \rrangle=
\int_0^1 f(s) V(s)\,ds=\left(\langle f, V_0\rangle, \dots,\langle f, V_n\rangle\right)\in \mathbb{R}^n.
\end{equation*}

Then the integrals appearing in the constraints in the variational
principle \ref{vp_with_z} are $\langle a^{[k]}_j, Q^{[k]}\rrangle$,
where $a^{[k]}_j$ are the polynomials
\[
a^{[k]}_j(s)=\frac{(1-s)^{k-j-1}}{(k-j-1)!},\quad j=0, \dots, k-1.
\]
These form a basis of the space of polynomials of degree at most
$k-1$. Let us consider a basis $b^{[k]}_j(s)$, $j=0, \dots, k-1$, of
the same space of polynomials consisting of orthonormal polynomials
on $[0,1]$, and let $(\gamma^{[k],i}_j)$, where $i,j=0, \dots, k-1$,
be the invertible real matrix such that
$a^{[k]}_j(s)=\gamma^{[k],i}_jb^{[k]}_i(s)$. For example, for $k=2$,
\[
a^{[2]}_0(s)=1-s,\quad a^{[2]}_1(s)=1,
\]
and we can take for instance the orthonormal basis
\[
b^{[2]}_0(s)=\sqrt{3}(1-2s),\quad b^{[2]}_1(s)=1;
\]
therefore,
\[
\gamma^{[2],0}_0=\textstyle\frac{1}{2\sqrt{3}},\quad
\gamma^{[2],1}_0=\frac{1}{2},\quad
\gamma^{[2],0}_1=0,\quad
\gamma^{[2],1}_1=1.
\]

Using this matrix, the constraints can be rewritten as
\begin{equation*}
z^{[j]}=\langle a^{[k]}_j,Q^{[k]}\rrangle=\gamma^{[k],i}_j\langle b^{[k]}_i(s),Q^{[k]}\rrangle,
\end{equation*}
for $j=0, \dots, k-1$. This allows us to reformulate the variational
principle in an equivalent way by replacing the data $(z^{[0]},
\dots, z^{[k-1]})$ and constraints $\langle
a^{[k]}_j,Q^{[k]}\rrangle=z^{[j]}$ by new data $(w^{[0]}, \dots,
w^{[k-1]})$ and constraints $\langle
b^{[k]}_j,Q^{[k]}\rrangle=w^{[j]}$, $j=0, \dots, k-1$. The old and
new data are related by
\begin{equation}\label{cambio de base}
\sum_{i=0}^{k-1} \gamma^{[k],i}_j w^{[i]}=z^{[j]}.
\end{equation}

\begin{vp}\label{vp_with_w}
Given $h$, $(q^{[0]}_1, \dots, q^{[k-1]}_1)$ and $(w^{[0]}, \dots,
w^{[k-1]})$, find a continuous curve $Q^{[k]}\colon [0,1] \to
\mathbb{R}^n$ that is a critical point of
\[
S_h=\int_0^1L\left(Q^{[0]}(u), \dots, Q^{[k]} (u)\right)\,du,
\]
where $Q^{[0]}(u)$, \dots, $Q^{[k-1]}(u)$ are defined by
\begin{equation}\label{defofQju}
Q^{[j]}(u)=q^{[j]}_1+\sum_{i=1}^{k-j-1}\frac{h^iu^i}{i!}q^{[j+i]}_1+h^{k-j}\int_0^u \frac{(u-s)^{k-j-1}}{(k-j-1)!}Q^{[k]}(s)\,ds,
\end{equation}
subject to the constraints
\begin{equation*}
  \int_0^1 b^{[k]}_j(s)Q^{[k]}(s)\,ds=w^{[j]},\quad j=0, \dots, k-1.\qedhere
\end{equation*}
\end{vp}

\subsection{Solution of the regularized
problem}\label{solregularized}

Let $S_h$ be given as in the variational principle \ref{vp_with_w},
regarded as a real-valued map defined on the Banach space
$C^0([0,1],\mathbb{R}^n)$ of curves $Q^{[k]}(u)$. We can also consider its restriction to the Banach space $C^k([0,1],\mathbb{R}^n)$. We are going to
use the following lemma \cite{AbMarsdRat}.

\begin{lemma}[Omega Lemma]
Let $E,F$ be Banach spaces, $U$ open in $E$, and $M$ a compact
topological space. Let $g\colon U \to F$ be a $C^r$ map, $r>0$. The
map
\[
\Omega_g\colon C^0(M,U) \to C^0(M,F)\quad \text{defined by}\quad \Omega_g(f)=g\circ f
\]
is also $C^r$, and $D\Omega_g(f)\cdot h=[(Dg)\circ f]\cdot h$.
\end{lemma}

The objective $S_h$ is the composition of the maps
\begin{center}\leavevmode
\xymatrix{ C^0([0,1], \mathbb{R}^n)\ar[r]^-i&
C^0([0,1],T^{(k)}Q)\ar[r]^-{\Omega_L}&
C^0([0,1],\mathbb{R})\ar[r]^-{\int}&\mathbb{R} }
\end{center}
where $i$ is defined by $Q^{[k]}(u) \mapsto (Q^{[0]}(u), \dots,
Q^{[k]} (u))$. Here $Q^{[0]}(u),\dots$, $Q^{[k-1]} (u)$ stand for
the right-hand sides of \eqref{defofQju}. Both $i$ and $\int$ are
bounded affine and therefore $C^\infty$. By the Omega Lemma,
$\Omega_L$ is $C^{k+1}$ because $L$ is $C^{k+1}$, and therefore so is $S_h$.

If we regard $S_h$ as defined on $C^k([0,1],\mathbb{R}^n)$, we should append the inclusion $C^k([0,1],\mathbb{R}^n)\hookrightarrow C^0([0,1],\mathbb{R}^n)$ to the left side of the diagram above. This inclusion is $C^\infty$ because it is linear and bounded ($\|Q^{[k]}\|_{C^0}\leq \|Q^{[k]}\|_{C^k}$ for all $Q^{[k]}$). Then $S_h$ is $C^{k+1}$ also as a map defined on $C^k([0,1],\mathbb{R}^n)$. In order to cover both cases, from now on $l$ will denote $0$ or $k$ interchangeably.

We need a suitable notion of the gradient of $S_h$, in order to find where it is perpendicular to the constraint space. In order to do that, let us first compute $\mathbf{d}S_h[Q^{[k]}(u)]$, for $Q^{[k]}$ of class $C^l$. The functions $Q^{[0]}(u)$,
\dots, $Q^{[k-1]}(u)$ are defined by \eqref{defofQju}. Since $S_h$ is smooth, we will compute $\mathbf{d}S_h$ using directional derivatives. For an arbitrary $\delta Q^{[k]}$ of class $C^l$, take a
deformation $Q^{[k]}_\epsilon(u)=Q^{[k]}(u)+\epsilon \delta
Q^{[k]}(u)$ of $Q^{[k]}(u)$. For $j=0, \dots, k-1$, define the
corresponding lower order curves as in  \eqref{defofQju} by
\begin{equation}\label{Qjepsilon}
Q^{[j]}_\epsilon(u)=
q^{[j]}_1+\sum_{i=1}^{k-j-1}\frac{h^iu^i}{i!}q^{[j+i]}_1+h^{k-j}\int_0^u \frac{(u-s)^{k-j-1}}{(k-j-1)!} Q^{[k]}_\epsilon(s)\,ds,
\end{equation}
so $Q^{[j]}_0(u)=Q^{[j]}(u)$ and
\begin{equation*}
\left.\frac{d}{d\epsilon}\right|_{\epsilon=0}Q^{[j]}_\epsilon(u)=
h^{k-j}\int_0^u \frac{(u-s)^{k-j-1}}{(k-j-1)!}\delta Q^{[k]}(s)\,ds.
\end{equation*}

Denoting $a^{[k]}_j(u,s)={(u-s)^{k-j-1}}/{(k-j-1)!}$ and
$Q(u)=(Q^{[0]}(u), \dots$, $Q^{[k]} (u))$ for short, we have
\begin{align*}
  &\mathbf{d}S_h[Q^{[k]}(u)]\cdot \delta Q^{[k]}(u)=\\
 &=\left.\frac{d}{d\epsilon}\right|_{\epsilon=0}\int_0^1L\left(Q^{[0]}_\epsilon(u), \dots, Q^{[k]}_\epsilon (u)\right)du\\
  &=\int_0^1 \left(\sum_{j=0}^{k-1} \frac{\partial L}{\partial q^{[j]}}(Q(u))h^{k-j}\int_0^u a^{[k]}_j(u,s) \delta Q^{[k]}(s)\,ds+\frac{\partial L}{\partial q^{[k]}}(Q(u))\delta Q^{[k]}(u)\right)du\\
 &=\sum_{j=0}^{k-1}\int_0^1\int_s^1\frac{\partial L}{\partial q^{[j]}}(Q(u))h^{k-j} a^{[k]}_j(u,s) \delta Q^{[k]}(s)\,du\,ds
 +\int_0^1 \frac{\partial L}{\partial q^{[k]}}(Q(u))\delta Q^{[k]}(u)\,du\\
 &=\sum_{j=0}^{k-1}\int_0^1\int_u^1\frac{\partial L}{\partial q^{[j]}}(Q(s))h^{k-j} a^{[k]}_j(s,u) \delta Q^{[k]}(u)\,ds\,du
 +\int_0^1 \frac{\partial L}{\partial q^{[k]}}(Q(u))\delta Q^{[k]}(u)\,du\\
 &=\int_0^1\left(\sum_{j=0}^{k-1}\int_u^1\frac{\partial L}{\partial q^{[j]}}(Q(s))h^{k-j} a^{[k]}_j(s,u)\,ds
+ \frac{\partial L}{\partial q^{[k]}}(Q(u))\right)\delta Q^{[k]}(u)\,du.\\
\end{align*}

For each $u\in[0,1]$, the first factor in the integrand of the last
expression is in $(\R^{n})^{*}$. If $\sharp\colon (\R^{n})^{*}\to\R^{n}$
denotes the index raising operator associated to the Euclidean inner
product, define
\begin{equation*}
\nabla S_h[Q^{[k]}(u)](u):=\left(\sum_{j=0}^{k-1}\int_u^1\frac{\partial L}{\partial q^{[j]}}(Q(s))h^{k-j} a^{[k]}_j(s,u)\,ds
+ \frac{\partial L}{\partial q^{[k]}}(Q(u))\right)^{\sharp}.
\end{equation*}

Since ${\partial L}/{\partial q^{[0]}}$, \dots, ${\partial
L}/{\partial q^{[k]}}$ are $C^{k}$ and the curve $Q$ is $C^l$ ($l=0$ or $l=k$) , then $\nabla S_h[Q^{[k]}(u)]$
is $C^l([0,1],\mathbb{R}^n)$.  Then we have a vector
field
\[\nabla S_h\colon C^l([0,1],\mathbb{R}^n)\to
C^{l}([0,1],\mathbb{R}^n)\] which we call the gradient of $S_h$. By the Omega Lemma, $\nabla S_h$ is a $C^k$ map.

Let us now compute the tangent space to the constraint set. If we consider the inner product on
$C^{l}([0,1],\mathbb{R}^n)$ given by
\[
\llangle V,W \rrangle=\int_0^1 V(u)\cdot W(u)\,du,
\]
then
\[\mathbf{d}S_h[Q^{[k]}(u)]\cdot \delta Q^{[k]}(u)=\llangle
\nabla S_h[Q^{[k]}(u)],\delta Q^{[k]}(u)\rrangle.\]

The constraints $g_j[Q^{[k]}(s)]:=\langle
b^{[k]}_j,Q^{[k]}\rrangle=w^{[j]}$, $j=0, \dots, k-1$, in the
variational principle \ref{vp_with_w} are bounded and linear, and
therefore $C^\infty$, and the corresponding derivatives are the same
functions $g_j$. Define
\[g=(g_0, \dots, g_{k-1})\colon C^l([0,1],\mathbb{R}^n) \to (\mathbb{R}^n)^k\]
so
\[
E=\Ker g \subset C^l([0,1],\mathbb{R}^n)
\]
is the tangent space to the constraint set. They are actually parallel since the constraints are linear. It is not difficult to show
using the definitions that the space
\[
E^\perp=\{c^jb^{[k]}_j \,|\, c^0, \dots, c^{k-1}\in \mathbb{R}^n\}
\]
of $\mathbb{R}^n$-valued polynomials of degree at most $k-1$ is
indeed the $\llangle,\rrangle$-orthogonal complement of $E$, which
is then a split subspace (see the Appendix for a proof). The orthogonal projection $P\colon
C^l([0,1],\mathbb{R}^n)=E\oplus E^\perp \to E$ is given by
\[
P(\delta Q^{[k]}(u))=\delta Q^{[k]}(u)-\sum_{j=0}^{k-1}\langle b^{[k]}_j, \delta Q^{[k]}\rrangle \, b^{[k]}_j.
\]

Now $S_h$ has a critical point on the constraint set (for any value
of the constraints) if and only if the projection $P\nabla S_h$ of
$\nabla S_h$ to the tangent space $E$ of the constraint set is $0$.
That is, in order to find solutions to the variational principle
\ref{vp_with_w}, we solve
\[
P\nabla S_h(Q^{[k]})=P\nabla S_h(Q^{[k]}_E\oplus Q^{[k]}_{E^\perp})=0
\]
for $Q^{[k]}_E$, near
\begin{align*}
Q^{[k]}=0,\quad w^{[0]}= \dots = w^{[k-1]}=0,\\
q^{[0]}_1=\bar q^{[0]}, \dots, q^{[k-1]}_1=\bar q^{[k-1]},\quad h=0.
\end{align*}

This can be solved using the implicit function theorem by requiring
that the partial derivative of $P\nabla S_h(Q^{[k]})$ at
the point $Q^{[k]}=0$ with respect to the space $E$ is a linear isomorphism. The variables $w^{[0]}, \dots, w^{[k-1]}$, $q^{[0]}_1, \dots, q^{[k-1]}_1$ and $h$ are seen as parameters that can move in some neighborhood. Note that it is not necessary to solve for $Q^{[k]}_{E^\perp}$ since it is completely determined by $w^{[0]}, \dots, w^{[k-1]}$ using the constraint equations in variational principle \ref{vp_with_w}.

In order to compute this partial
derivative, take a deformation of $Q^{[k]}=0$ of the form
$Q^{[k]}_\epsilon=\epsilon \delta Q^{[k]}_E$, where $\delta
Q^{[k]}_E\in E$. Recalling \eqref{Qjepsilon} and noting that $h=0$,
we have
\begin{align*}
\left.\frac{d}{d\epsilon}\right|_{\epsilon=0}&P\frac{\partial L}{\partial q^{[k]}}(Q^{[0]}_\epsilon(u), \dots, Q^{[k]}_\epsilon(u))=\left.\frac{d}{d\epsilon}\right|_{\epsilon=0}P\frac{\partial L}{\partial q^{[k]}}(\bar q^{[0]}, \dots, \bar q^{[k-1]},Q^{[k]}_\epsilon(u))\\
&=P\frac{\partial^2 L}{\partial q^{[k]2}}(\bar q^{[0]}, \dots, \bar q^{[k-1]},0)\delta Q^{[k]}_E(u)=\frac{\partial^2 L}{\partial q^{[k]2}}(\bar q^{[0]}, \dots, \bar q^{[k-1]},0)\delta Q^{[k]}_E(u)\\
&\quad
-\sum_{j=0}^{k-1}\Big\langle b^{[k]}_j,
\frac{\partial^2 L}{\partial q^{[k]2}}(\bar q^{[0]}, \dots, \bar q^{[k-1]},0)\delta Q^{[k]}_E
\Bigrrangle \, b^{[k]}_j=\frac{\partial^2 L}{\partial q^{[k]2}}(\bar q^{[0]}, \dots, \bar q^{[k-1]},0)\delta Q^{[k]}_E(u).
\end{align*}

Here the inner products vanish because $\frac{\partial^2 L}{\partial
q^{[k]2}}(\bar q^{[0]}, \dots, \bar q^{[k-1]},0)$ is a constant
matrix (that is, it does not depend on $u$) and $\langle
b^{[j]},\delta Q^{[k]}_E\rrangle=0$ for $j=0, \dots, k-1$.

Then the derivative is precisely $\frac{\partial^2 L}{\partial
q^{[k]2}}(\bar q^{[0]}, \dots, \bar q^{[k-1]},0)$, seen as a linear
map from $E$ into itself, and if $L$ is regular then it is an
isomorphism.

By the implicit function theorem, there are neighborhoods
$W_1\subseteq (\mathbb{R}^n)^k\times (\mathbb{R}^n)^k\times
\mathbb{R}$ (with variables $(q^{[0]}_1, \dots, q^{[k-1]}_1;w^{[0]},
\dots, w^{[k-1]};h)$) containing $(\bar q^{[0]}$, $\dots, \bar
q^{[k-1]};0, \dots, 0;0)$ and $W_2^{l}\subseteq
C^l([0,1],\mathbb{R}^n)$ containing the constant curve
$Q^{[k]}(u)=0$, and a $C^{k}$ map $\psi\colon W_1 \to W_2^{l}$
such that for each $(q^{[0]}_1, \dots, q^{[k-1]}_1;w^{[0]}$, $\dots,
w^{[k-1]};h)\in W_1$, the curve
\[
Q^{[k]}=\psi(q^{[0]}_1, \dots, q^{[k-1]}_1;w^{[0]}, \dots, w^{[k-1]};h)\in C^l([0,1],\mathbb{R}^n)
\]
is the unique critical point in $W_2^{l}$ of the variational problem
\ref{vp_with_w}. Thus, $\psi$ maps initial conditions, constraint values (which encode the final endpoint conditions for the original problem) and $h$ into $C^l$ curves.  

Let us now consider the cases $l=0$ and $l=k$ separately. Taking $l=k$, $\psi$ has values in $W_2^{k}\subseteq
C^{k}([0,1],\mathbb{R}^n)$. Taking $l=0$, $\psi$ has values in
$W_2^{0}\subseteq C^{0}([0,1],\mathbb{R}^n)$. However, since
$C^{k}([0,1],\mathbb{R}^n)\subset C^{0}([0,1],\mathbb{R}^n)$, this $\psi$ also provides the unique
solution among the $C^0$ curves in a $C^0$-open neighborhood of the
curve $u \mapsto 0$, say $\{Q^{[k]}(u)\,|\,\|Q^{[k]}\|_0<\epsilon\}$.

Let us now reverse the regularization in order to obtain a unique
$C^{2k}$ solution of the variational principle \ref{originalvp}. Let $h\neq 0$. For
$(q_1,q_2)=((q_{1}^{[0]},\ldots,q_{1}^{[k-1]}),(q_{2}^{[0]},\ldots,q_{2}^{[k-1]}))\in
(\R^{n})^{k}\times (\R^{n})^{k}$ the corresponding values of
$z^{[0]},\ldots,z^{[k-1]}$ are given by \eqref{def_of_z} and the
values of $w^{[0]},\ldots,w^{[k-1]}$ can be computed from
\eqref{cambio de base} using the inverse matrix of
$\left(\gamma_{j}^{[k],i}\right)$. This defines a smooth function
$(w^{[0]},\ldots,w^{[k-1]})=\varpi(q_1,q_2,h)$. Note that the condition that $q_1$ and $q_2$ are close translates into the condition that $(w^{[0]},\ldots,w^{[k-1]})$ is close to 0.

Let $h>0$ be such that $(\bar q^{[0]}, \dots, \bar q^{[k-1]};0,\dots,0;h)\in W_1$. Define
\[
\widetilde W_1=\{(q_1,q_2)\in (\mathbb{R}^n)^k \times (\mathbb{R}^n)^k\,|\, (q_1;\varpi(q_1,q_2,h);h)\in W_1\}
\]
and for each $(q_1,q_2)=\left((q^{[0]}_1, \dots,
q^{[k-1]}_1),(q^{[0]}_2, \dots,  q^{[k-1]}_2)\right)\in W_1$ define
the curve $Q^{[0]}_{(q_1,q_2)}(u)$ according to
\eqref{Qj_in_terms_of_Qk} as
\begin{equation*}
Q^{[0]}_{(q_1,q_2)}(u)=\sum_{i=0}^{k-1}\frac{h^iu^i}{i!}q^{[i]}_1+h^{k}\int_0^u \frac{(u-s)^{k-1}}{(k-1)!} \psi\left(q_1;\varpi(q_1,q_2,h);h\right)(s)\,ds.
\end{equation*}
Since $\psi$ takes values in the $C^k$ curves,
$Q^{[0]}_{(q_1,q_2)}(u)$ is $C^{2k}$ by the reasoning leading to
equation \eqref{Qj_in_terms_of_Qk}.

Now reparameterize with $t=hu$ to get a $C^{2k}$ curve
 \begin{equation*}
q^{[0]}_{(q_1,q_2)}(t)=\sum_{i=0}^{k-1}\frac{t^i}{i!}q^{[i]}_1+\left(\frac{t}{u}\right)^{k}\int_0^{t/h} \frac{(t/h-s)^{k-1}}{(k-1)!} \psi\left(q_1;\varpi(q_1,q_2,h);h\right)(s)\,ds
\end{equation*}
on $Q$, defined for $t\in [0,h]$. This curve is the unique solution
of the variational principle \ref{originalvp} with endpoint
conditions $q_1$ and $q_2$.

This solution is $C^{2k}$, and unique among the curves corresponding to $Q^{[k]}$
continuous with $\|Q^{[k]}\|_0<\epsilon$. These are the $C^k$ curves
$q(t)$ on $Q$ with $\|q^{(k)}\|_0<\epsilon/h^k$, which are the $C^k$
curves in some $C^k$ neighborhood of the constant curve $t \mapsto
\bar q^{[0]}$.

\section{The exact discrete Lagrangian and discrete equations for second-order systems}\label{section2}

Next, we will consider second-order Lagrangian systems, motivated by the study of optimal control problems. Let $Q$ be a configuration manifold and let $L\colon T^{(2)}Q\to\R$ be a
regular Lagrangian.

\begin{definition} Given a small enough\footnote{By this we mean, from now on, that there exists $h_0>0$ such that for all $h\in(0,h_0)$ the definition or proof holds.} $h>0$, the \textit{exact discrete lagrangian}
$L_d^{e}\colon TQ\times TQ\to\R$ is defined by
\[L_d^{e}(q_0,\dot{q}_0,q_1,\dot{q}_1)=\int_{0}^{h}L(q(t),\dot{q}(t),\ddot{q}(t))dt,\]
where $q\colon [0,h]\to Q$ is the unique solution of the
Euler--Lagrange equations for the second-order Lagrangian $L$,
\begin{equation*}
\frac{d^2}{dt^2}\frac{\partial L}{\partial\ddot{q}}-\frac{d}{dt}\frac{\partial L}{\partial\dot{q}}+\frac{\partial L}{\partial q}=0,
\end{equation*}
satisfying the boundary conditions
$q(0)=q_0,q(h)=q_1,\dot{q}(0)=\dot{q}_0$ and $\dot{q}(h)=\dot{q}_1$.
\end{definition}

Strictly speaking, the exact discrete Lagrangian is defined not on $TQ\times TQ$ but on a neighborhood of the diagonal. For the sake of simplicity, we will not make this distinction. Our idea is to take a discrete Lagrangian $L_{d}\colon TQ\times TQ\to\R$ as an approximation of $L_{d}^{e}\colon TQ\times TQ\to\R$, to construct variational integrators in
the same way as in discrete mechanics (see section
\ref{section3}). In other words, for given $h>0$ we define
$L_d(q_0,v_0,q_1,v_1)$ as an approximation of the action integral
along the exact solution curve segment $q(t)$ with boundary
conditions $q(0)=q_0$, $\dot{q}(0)=v_0$, $q(h)=q_1$, and
$\dot{q}(h)=v_1$. For example, we can use the formula
\[L_d(q_0,v_0,
q_1,v_1)=hL\left(\kappa(q_0,v_0,q_1,v_1),\chi(q_0,v_0,q_1,v_1),\zeta(q_0,v_0,q_1,v_1)\right),\]
where $\kappa$, $\chi$ and $\zeta$ are functions of
$(q_0,v_0,q_1,v_1)\in TQ\times TQ$ which approximate the
configuration $q(t)$, the velocity $\dot{q}(t)$ and
the acceleration $\ddot{q}(t)$, respectively, in terms
of the initial and final positions and velocities. We can also, for
instance, consider suitable linear combinations of discrete
Lagrangians of this type, for instance, weighted averages of the
type
\[L_{d}(q_0,v_0,q_1,v_1)=\frac{1}{2}L\left(q_0,v_0,\frac{v_1-v_0}{h}\right)+\frac{1}{2}L\left(q_1,v_1,\frac{v_1-v_0}{h}\right),\]
or other combinations.

For completeness, we will derive the discrete equations for the
Lagrangian $L_{d}\colon TQ\times TQ\to\R$, but these results are a direct
translation of Marsden and West \cite{mawest} to our case.

Given the grid $\{t_{k}=kh\mid k=0,\ldots,N\}$, $Nh=T$,  define
the discrete path space
$\mathcal{P}_{d}(TQ):=\{(q_{d},v_{d})\colon \{t_{k}\}_{k=0}^{N}\to TQ\}$.
We will identify a discrete trajectory
$(q_{d},v_{d})\in\mathcal{P}_{d}(TQ)$ with its image
$(q_{d},v_{d})=\{(q_{k},v_{k})\}_{k=0}^{N}$ where
$(q_{k},v_{k}):=(q_{d}(t_{k}),v_{d}(t_{k}))$. The discrete action
$\mathcal{A}_{d}\colon \mathcal{P}_{d}(TQ)\to\R$ along this sequence is
calculated by summing the discrete Lagrangian evaluated at each pair of adjacent points of the discrete path, that is,
\[\mathcal{A}_d(q_{d},v_{d}):=\sum_{k=0}^{N-1}L_d(q_k,v_k,q_{k+1},v_{k+1}).\]
We would like to point out that the discrete path space is
isomorphic to the smooth product manifold which consists on $N+1$
copies of $TQ$, the discrete action inherits the smoothness of the
discrete Lagrangian, and the tangent space
$T_{(q_{d},v_{d})}\mathcal{P}_{d}(TQ)$ at $(q_{d},v_{d})$ is the set
of maps $a_{(q_{d},v_{d})}\colon \{t_{k}\}_{k=0}^{N}\to TTQ$ such that
$\tau_{TQ}\circ a_{(q_{d},v_{d})}=(q_{d},v_{d})$ where
$\tau_{TQ}\colon TTQ\to TQ$ is the canonical projection.

Hamilton's principle seeks discrete curves
$\{(q_{k},v_{k})\}_{k=0}^{N}$ that satisfy
\[\delta\sum_{k=0}^{N-1}L_{d}(q_{k},v_{k},q_{k+1}, v_{k+1})=0\] for
all variations $\{(\delta q_{k},\delta v_{k})\}_{k=0}^{N}$ vanishing
at the endpoints. This is equivalent to the \textit{discrete
Euler--Lagrange equations}
\begin{subequations}\label{euler-lagrange equations}
\begin{align}
D_3L_d(q_{k-1},v_{k-1},q_k,v_k)+D_1L_d(q_k,v_k,q_{k+1},v_{k+1})=0,\\
D_4L_d(q_{k-1},v_{k-1},q_k,v_k)+D_2L_d(q_k,v_k,q_{k+1},v_{k+1})=0,
\end{align}
\end{subequations}
for $1\leq k\leq N-1$.

Given a solution $\{q_{k}^{*},v_{k}^{*}\}_{k\in\mathbb{Z}}$ of
equations \eqref{euler-lagrange equations} and assuming that the
$2n\times 2n$ matrix \[\left(
                \begin{array}{cc}
                  D_{13}L_d(q_{k},v_{k},q_{k+1},v_{k+1}) & D_{14}L_d(q_{k},v_{k},q_{k+1},v_{k+1}) \\
                  D_{23}L_d(q_{k},v_{k},q_{k+1},v_{k+1}) & D_{24}L_d(q_{k},v_{k},q_{k+1},v_{k+1}) \\
                \end{array}
              \right)\]
 is nonsingular, it is possible to define the (local) \textit{discrete flow} $F_{L_{d}}\colon \mathcal{U}_{k}\subset TQ\times TQ\to TQ\times TQ$
mapping $(q_{k-1},v_{k-1},q_{k},v_{k})$ to
$(q_{k},v_{k},q_{k+1},v_{k+1})$ from \eqref{euler-lagrange
equations} where $\mathcal{U}_{k}$ is a neighborhood of the point
$(q_{k-1}^{*},v_{k-1}^{*},q_{k}^{*},v_{k}^{*})$.
The simplecticity and momentum preservation of the discrete flow is derived in \cite{mawest}.

\begin{example}\label{example_cubic_splines}\textbf{Cubic splines}
Let $Q=\R^{n}$ and $L\colon T^{(2)}Q\equiv(\R^{n})^{3}\to\R$ be the
second-order Lagrangian given by
$L(q,\dot{q},\ddot{q})=\frac{1}{2}\ddot{q}^{2}$.

It is well known that the solutions to the corresponding
Euler--Lagrange equations $q^{(4)}=0$ are the so-called cubic splines
$q(t)=at^{3}+bt^{2}+ct+d$, for $a,b,c,d\in\R^{n}$. We define
$L_{d}\colon (\R^{n}\times\R^{n})\times(\R^{n}\times\R^{n})\to\R$ as
follows. Write \begin{subequations}\label{discrete Taylor}
\begin{align}
q(0)=q(h)-h\dot{q}(h)+\frac{h^{2}}{2}\ddot{q}(h)+\mathpzc{O}(h^{3}),\\
q(h)=q(0)+h\dot{q}(0)+\frac{h^2}{2}\ddot{q}(0)+\mathpzc{O}(h^{3}).
\end{align}
\end{subequations} Given sufficiently close $(q_0,v_0),(q_1,v_1)\in TQ$ we can use
equations \eqref{discrete Taylor} to obtain approximations of the
acceleration of the exact solution joining these boundary
conditions at time $h$, which we call
\[a_0=\frac{2}{h^{2}}(q_1-q_0-hv_0)\hbox{ and }
a_1=\frac{2}{h^{2}}(q_0-q_1+hv_1).\] Then we define

\[L_{d}(q_0,v_0,q_1,v_1)=\frac{h}{2}\left(L(q_0,v_0,a_0)+L(q_1,v_1,a_1)\right)=\frac{(hv_1+q_0-q_1)^{2}}{h^3}+\frac{(-hv_0-q_0+q_1)^2}{h^{3}}.\]

Solving the discrete second-order Euler--Lagrange equations for this
discrete Lagrangian, the evolution of the discrete trajectory is
\begin{subequations}\label{metodosplineTaylor}
\begin{align}
&q_{k+1}=q_{k-1}+2hv_k,\\
&v_{k+1}=v_{k-1}+4\left(v_{k}-\frac{q_{k}-q_{k-1}}{h}\right).
\end{align}
\end{subequations}
In the following section we will continue this example and show some simulations.
\end{example}

\subsection{Discrete Legendre transforms}
We define the \textit{discrete Legendre transforms}
$\mathbb{F}^{+}L_{d},\mathbb{F}^{-}L_{d}\colon TQ\times TQ\to T^{*}TQ$
which maps the space $TQ\times TQ$ into $T^{*}TQ$. These are given
by
\begin{align*}
\mathbb{F}^{+}L_{d}(q_{0},v_0,q_1,v_1)&=\left(q_0,v_0,-D_{1}L_{d}(q_0,v_0,q_1,v_1),-D_{2}L_{d}(q_0,v_0,q_1,v_1)\right),\\
\mathbb{F}^{-}L_{d}(q_{0},v_0,q_1,v_1)&=\left(q_1,v_1,D_{3}L_{d}(q_0,v_0,q_1,v_1),D_{4}L_{d}(q_0,v_0,q_1,v_1)\right).
\end{align*}
If both discrete fibre derivatives are locally diffeomorphisms for
nearby $(q_0,v_0)$ and $(q_1,v_1)$, then we say that $L_d$ is
\textit{regular}. 

Using the discrete Legendre transforms the discrete Euler--Lagrange
equations \eqref{euler-lagrange equations} can be rewritten as
\[\mathbb{F}^{-}L_{d}(q_k,v_k,q_{k+1},v_{k+1})=\mathbb{F}^{+}L_{d}(q_{k-1},v_{k-1},q_k,v_k).\]

It will be useful to note that
\begin{align*}
\mathbb{F}^{-}L_{d}\circ F_{L_d}(q_0,v_0,q_1,v_1)&=\mathbb{F}^{-}L_{d}(q_1,v_1,q_2,v_2)\\
&=\left(q_1,v_1,-D_{1}L_{d}(q_1,v_1,q_2,v_2),-D_{2}L_{d}(q_1,v_1,q_2,v_2)\right)\\
&=\left(q_1,v_1,D_{3}L_{d}(q_0,v_0,q_1,v_1),D_{4}L_{d}(q_0,v_0,q_1,v_1)\right)\\
&=\mathbb{F}^{+}L_{d}(q_0,v_0,q_1,v_1),
\end{align*} that is, \begin{equation}\label{relationF}\mathbb{F}^{+}L_{d}=\mathbb{F}^{-}L_{d}\circ F_{L_{d}}.\end{equation}

\begin{remark} It is easy to extend this framework to
higher-order mechanical systems. Let $L\colon T^{(\ell)}Q\to\R$ be a regular higher-order Lagrangian.
Given a small enough $h>0$, the \textit{exact discrete Lagrangian} $L_d^{e}\colon T^{(\ell-1)}Q\times
T^{(\ell-1)}Q\to\R$ is defined by
\[L_d^{e}(q_0^{(0)},q_0^{(1)},\ldots,q_0^{(\ell-1)};q_1^{(0)},q_1^{(1)},\ldots,q_1^{(\ell-1)})=\int_{0}^{h}L(q(t),\dot{q}(t),\ldots,q^{(\ell)}(t))dt,\]
where $q(t)\colon I\subset\R\to Q$ is the unique solution of the
Euler--Lagrange equations for the higher-order Lagrangian $L$,
\[ \sum_{j=0}^{\ell}(-1)^j
\frac{d^j}{dt^j}\frac{\partial L}{\partial q^{(j)}}=0,
\]
satisfying the boundary conditions
$q(0)=q_0^{(0)},\dot{q}(0)=q_0^{(1)},\ldots,q^{(\ell-1)}(0)=q_0^{(\ell-1)},q(h)=q_1^{(0)},\dot{q}(h)=q_1^{(1)},\ldots,q^{(\ell-1)}(h)=q_1^{(\ell-1)}$.

The exact discrete Lagrangian is actually defined on a neighborhood of the diagonal of $T^{(\ell-1)}Q\times T^{(\ell-1)}Q$. We take 
$L_{d}\colon T^{(\ell-1)}Q\times T^{(\ell-1)}Q\to\R$ to be an approximation of
$L_{d}^{e}$ in order to construct
variational integrators for higher-order mechanical systems.

Given a discrete path $\{(q_k^{(0)},\dots,q_k^{(\ell-1)})\in T^{(\ell-1)}Q\}|_{k=0}^N$, the corresponding discrete action is defined as
\[\mathcal{A}_d:=\sum_{k=0}^{N-1}L_d(q_k^{(0)},\dots,q_k^{(\ell-1)};q_{k+1}^{(0)},\dots,q_{k+1}^{(\ell-1)}).\]

Hamilton's principle seeks discrete paths that satisfy $\delta \mathcal{A}_d=0$  for all variations $\{(\delta q_k^{(0)},\dots,\allowbreak \delta q_k^{(\ell-1)})|_{k=0}^N\}$ vanishing at the endpoints $k=0,N$. This is equivalent to the \textit{discrete higher-order
Euler--Lagrange equations for $L_{d}$:}
\begin{equation*}
D_{i+\ell}L_d(q_{k-1}^{(0)},\dots,q_{k-1}^{(\ell-1)};q_{k}^{(0)},\dots,q_{k}^{(\ell-1)})+D_{i}L_d(q_{k}^{(0)},\dots,q_{k}^{(\ell-1)};q_{k+1}^{(0)},\dots,q_{k+1}^{(\ell-1)})=0
\end{equation*} for $i=1,\dots,\ell$ and $k=1,\dots,N-1$. 
\end{remark}

\section{Relationship between discrete and continuous variational
systems}\label{section3}
Let $L\colon T^{(2)}Q\to\R$ be a regular Lagrangian and, for small enough $h>0$, consider the exact
discrete Lagrangian defined before, that is, a function $L_d^{e}\colon TQ\times TQ\to\R$ given by
\[L_d^{e}(q_0,\dot{q}_0,q_1,\dot{q}_1)=\int_{0}^{h}L(q(t),\dot{q}(t),\ddot{q}(t))dt,\]
where $q\colon [0,h]\to Q$ is the unique solution of the
Euler--Lagrange equations for the second-order Lagrangian $L$,
\begin{equation*}
\frac{d^2}{dt^2}\frac{\partial L}{\partial\ddot{q}}-\frac{d}{dt}\frac{\partial L}{\partial\dot{q}}+\frac{\partial L}{\partial q}=0
\end{equation*}
satisfying the boundary conditions
$q(0)=q_0,q(h)=q_1,\dot{q}(0)=\dot{q}_0$ and $\dot{q}(h)=\dot{q}_1$.

The Legendre transformation associated to $L$  is defined to be the
map $\mathbb{F}L\colon T^{(3)}Q\to T^{*}TQ$ given by (see \cite{LR1})
\begin{equation*}
\mathbb{F}L(q,\dot{q},\ddot{q},{q}^{(3)})=\left(q,\dot{q},\frac{\partial L}{\partial\dot{q}}-\frac{d}{dt}\frac{\partial{L}}{\partial\ddot{q}},\frac{\partial L}{\partial\ddot{q}}\right).
\end{equation*}

We will see that there is a special relationship between the
Legendre transform of a regular Lagrangian and the discrete Legendre
transforms of the corresponding exact discrete Lagrangian $L_d^{e}$.

\begin{theorem}\label{Theo1}
Let $L\colon T^{(2)}Q\to\R$ be a regular Lagrangian and $L_d^{e}\colon TQ\times
TQ\to\R$, the corresponding exact discrete Lagrangian. Then $L$ and $L_d^{e}$
have Legendre transformations related by
\begin{align*}
\mathbb{F}^{-}L_d^{e}(q(0),\dot{q}(0),q(h),\dot{q}(h))&=\F L(q(0),\dot{q}(0),\ddot{q}(0), {q}^{(3)}(0))\\
\mathbb{F}^{+}L_d^{e}(q(0),\dot{q}(0),q(h),\dot{q}(h))&=\F L(q(h),\dot{q}(h),\ddot{q}(h), {q}^{(3)}(h)),
\end{align*}
where $q (t)$ is a solution of the second-order Euler--Lagrange equations. 
\end{theorem}

\begin{proof}
We begin by computing the derivatives of $L_d^{e}$.
\begin{align*}\frac{\partial L_d^{e}}{\partial
q_0}&=\int_{0}^{h}\left(\frac{\partial L}{\partial q}\frac{\partial
q}{\partial q_0}+\frac{\partial
L}{\partial\dot{q}}\frac{\partial\dot{q}}{\partial
q_0}+\frac{\partial
L}{\partial\ddot{q}}\frac{\partial\ddot{q}}{\partial
q_0}\right)dt\\
&=\int_0^{h}\left(\frac{\partial L}{\partial q}\frac{\partial
q}{\partial q_0}+\frac{\partial
L}{\partial\dot{q}}\frac{\partial\dot{q}}{\partial
q_0}-\left(\frac{d}{dt}\frac{\partial
L}{\partial\ddot{q}}\right)\frac{\partial\dot{q}}{\partial
q_0}\right)dt+\left(\frac{\partial L}{\partial \ddot{q}}\frac{\partial\dot{q}}{\partial q_0}\right)\Big{|}_{0}^{h}\\
&=\int_0^{h}\left( \frac{\partial L}{\partial q}\frac{\partial
q}{\partial q_0}+\left(\frac{\partial
L}{\partial\dot{q}}-\frac{d}{dt}\frac{\partial
L}{\partial\ddot{q}}\right)\frac{\partial\dot{q}}{\partial
q_0}\right)dt,
\end{align*}
where we have used integration by parts and the fact that \[\frac{\partial\dot{q}}{\partial q_0}(0)=0\hbox{ and }
\frac{\partial\dot{q}}{\partial q_0}(h)=0.\]
Therefore,
\[\frac{\partial L_d^{e}}{\partial
q_0}=\left(\left(\frac{\partial
L}{\partial\dot{q}}-\frac{d}{dt}\frac{\partial
L}{\partial\ddot{q}}\right)\frac{\partial{q}}{\partial
q_0}\right)\Big{|}_{0}^{h}+\int_{0}^{h}\left(\frac{\partial L}{\partial
q}-\frac{d}{dt}\frac{\partial
L}{\partial\dot{q}}+\frac{d^2}{dt^2}\frac{\partial
L}{\partial\ddot{q}}\right)\frac{\partial q}{\partial
q_0}\,dt.\]
Since $q(t)$ is a solution of the Euler--Lagrange equations
for $L\colon T^{(2)}Q\to\R$, the last term  is zero.  Therefore,
\begin{equation}\label{dLq0}\frac{\partial L_d^{e}}{\partial
q_0}=\left(\left(\frac{\partial L}{\partial\dot{q}}-\frac{d}{dt}\frac{\partial
L}{\partial\ddot{q}}\right)\frac{\partial q}{\partial q_0}\right)\Big{|}_{0}^{h}=
\left(-\frac{\partial
L}{\partial\dot{q}}+\frac{d}{dt}\frac{\partial
L}{\partial\ddot{q}}\right)(q(0),\dot{q}(0),\ddot{q}(0), q^{(3)}(0)),
\end{equation} because \[\frac{\partial q}{\partial q_0}(0)=\operatorname{Id} \hbox{ and }
\frac{\partial q}{\partial q_0}(h)=0.\]

On the other hand,
\begin{align*}\frac{\partial L_d^{e}}{\partial
\dot{q}_0}&=\int_{0}^{h}\left(\frac{\partial L}{\partial q}\frac{\partial
q}{\partial \dot{q}_0}+\frac{\partial
L}{\partial\dot{q}}\frac{\partial\dot{q}}{\partial
\dot{q}_0}+\frac{\partial
L}{\partial\ddot{q}}\frac{\partial\ddot{q}}{\partial
\dot{q}_0}\right)dt=\\
&\int_0^{h}\left(\frac{\partial L}{\partial q}\frac{\partial
q}{\partial \dot{q}_0}+\frac{\partial
L}{\partial\dot{q}}\frac{\partial\dot{q}}{\partial
\dot{q}_0}-\left(\frac{d}{dt}\frac{\partial
L}{\partial\ddot{q}}\right)\frac{\partial\dot{q}}{\partial
\dot{q}_0}\right)dt+\left(\frac{\partial L}{\partial \ddot{q}}\frac{\partial\dot{q}}{\partial \dot{q}_0}\right)\Big{|}_{0}^{h}=\\
&\int_0^{h}\left( \frac{\partial L}{\partial q}\frac{\partial
q}{\partial \dot{q}_0}+\left(\frac{\partial
L}{\partial\dot{q}}-\frac{d}{dt}\frac{\partial
L}{\partial\ddot{q}}\right)\frac{\partial\dot{q}}{\partial
\dot{q}_0}\right)dt+\left(\frac{\partial L}{\partial \ddot{q}}\frac{\partial\dot{q}}{\partial \dot{q}_0}\right)\Big{|}_{0}^{h}=\\
&\int_{0}^{h}\left(\frac{\partial
L}{\partial q}-\frac{d}{dt}\frac{\partial
L}{\partial\dot{q}}+\frac{d^2}{dt^2}\frac{\partial L}{\partial\ddot{q}}\right)\frac{\partial q}{\partial
\dot{q}_0}dt+\frac{\partial L}{\partial \ddot{q}}\frac{\partial\dot{q}}{\partial \dot{q}_0}\Big{|}_{0}^{h}+\left(\frac{\partial L}{\partial\dot{q}}-\frac{d}{dt}\frac{\partial
L}{\partial\ddot{q}}\right)\frac{\partial q}{\partial \dot{q}_0}\Big{|}_{0}^{h}.
\end{align*}
Since $q(t)$ is a solution of the Euler--Lagrange equations, the first term is zero, and using that
 \[\frac{\partial\dot{q}}{\partial \dot{q}_0}(0)=\operatorname{Id},\quad \frac{\partial\dot{q}}{\partial
 \dot{q}_0}(h)=0,\quad \frac{\partial q}{\partial \dot{q}_0}(0)=0, \hbox{ and }
\frac{\partial q}{\partial \dot{q}_0}(h)=0,\]
we have \[\frac{\partial L_d^{e}}{\partial
\dot{q}_0}=-\frac{\partial
L}{\partial\ddot{q}}(q(0),\dot{q}(0),\ddot{q}(0)).\]
Therefore
\begin{align*}
  \mathbb{F}^{-}L_d^{e}(q(0),\dot{q}(0),q(h),\dot{q}(h))&=\Big(q(0), \dot q(0),-\frac{\partial L_d^{e}}{\partial
{q}_0}(q(0),\dot{q}(0),q(h),\dot{q}(h)),-\frac{\partial L_d^{e}}{\partial
\dot{q}_0}(q(0),\dot{q}(0),q(h),\dot{q}(h))\Big)\\
&=\F L(q(0),\dot{q}(0),\ddot{q}(0),
q^{(3)}(0)).
\end{align*}

With similar arguments, we can also prove that
\[\frac{\partial L_d^{e}}{\partial
q_1}=\left(\frac{\partial
L}{\partial\dot{q}}-\frac{d}{dt}\frac{\partial
L}{\partial\ddot{q}}\right)(q(h),\dot{q}(h),\ddot{q}(h), q^{(3)}(h))\] and
\[\frac{\partial L_d^{e}}{\partial
\dot{q}_1}=\frac{\partial
L}{\partial\ddot{q}}(q(h),\dot{q}(h),\ddot{q}(h)),\] and in
consequence,
\[\mathbb{F}^{+}L_d^{e}(q(0),\dot{q}(0),q(h),\dot{q}(h))=\F
L(q(h),\dot{q}(h),\ddot{q}(h), q^{(3)}(h)).\qedhere\]
\end{proof}

In what follows we will study the relation between the regularity of the  continuous Lagrangian, given by the hessian matrix
\[
{\mathcal W}=\left(\frac{\partial^2 L}{\partial \ddot{q}\; \partial \ddot{q}}\right)
\]
and the regularity condition corresponding to the exact discrete Lagrangian $L_d^e\colon  TQ\times TQ\to {\mathbb R}$
\[
{\mathcal W}_d=\left(
\begin{array}{rr}
D_{13}L_d^e&D_{14} L_d^e\\
D_{23}L_d^e&D_{24}L_d^e
\end{array}
\right).
\]

For the next theorem, we restrict ourselves to Lagrangians that can be written locally as
\begin{equation}\label{special_Lagrangian}
  L(q, \dot q, \ddot q)= \frac{1}{2}g_{ij}(q) \ddot q^i \ddot q^j+ \ddot q^if_i(q, \dot q)+ V(q, \dot q),
\end{equation}
where $(g_{ij}(q))$ is a regular matrix for all $q$. It is also possible to write this condition intrinsically by using a metric, a connection, a one-form and a function. This covers the kind of Lagrangians that appear in interpolation problems \cite{MR2864799} and in optimal control problems with cost functionals of the form $\frac{1}{2}\int_0^T\|u\|^2dt$, where $u$ represents the control force applied to a system having a (first-order) Lagrangian of mechanical type (see section \ref{section5}).

\begin{theorem}
Let $L\colon T^{(2)}Q \to \mathbb{R}$ be a regular Lagrangian of the type \eqref{special_Lagrangian}. For small enough $h>0$, the corresponding exact discrete Lagrangian $L_d^e\colon TQ\times TQ\to {\mathbb R}$ is also regular.
\end{theorem}
\begin{proof} We will work locally. Given $q_0$, $ \dot q_0$, $q_1$, $ \dot q_1$, consider the curve $q(t)$ that solves the Euler--Lagrange equations with those boundary values, as in the definition of $L_d^e$. Using the Taylor expansions for $q(t)$ and $ \dot q(t)$, we can write
\begin{align*}
q(h)&=q(0)+h\dot{q}(0)+\frac{h^2}{2}\ddot{q}(0)+\frac{h^3}{6}q^{(3)}(0)+{\mathpzc O}\left(h^4\right),\\
\dot{q}(h)&=\dot{q}(0)+h\ddot{q}(0)+\frac{h^2}{2}q^{(3)}(0)+{\mathpzc O}\left(h^3\right),
\end{align*}
for $h\to 0$. 
By differentiating these expressions with respect to the parameters $q_0$ and $\dot q_0$, we get two systems of equations from which we find
\begin{alignat*}{2}
\frac{\partial \ddot{q}}{\partial {q}_0}(h)&=\frac{6}{h^2}\operatorname{Id}+{\mathpzc O}\left(h^2\right) ,\quad&
\frac{\partial q^{(3)}}{\partial {q}_0}(h)&=\frac{12}{h^3}\operatorname{Id}+{\mathpzc O}\left(h\right),\\
\frac{\partial \ddot{q}}{\partial \dot{q}_0}(h)&=\frac{2}{h}\operatorname{Id}+{\mathpzc O}\left(h^2\right) ,\quad&
\frac{\partial q^{(3)}}{\partial \dot{q}_0}(h)&=\frac{6}{h^2}\operatorname{Id}+{\mathpzc O}\left(h\right).
\end{alignat*}
Analogously,
\begin{alignat*}{2}
\frac{\partial \ddot{q}}{\partial {q}_1}(0)&=\frac{6}{h^2}\operatorname{Id}+{\mathpzc O}\left(h^2\right) ,\quad&
\frac{\partial q^{(3)}}{\partial {q}_1}(0)&=-\frac{12}{h^3}\operatorname{Id}+{\mathpzc O}\left(h\right),\\
\frac{\partial \ddot{q}}{\partial \dot{q}_1}(0)&=-\frac{2}{h}\operatorname{Id}+{\mathpzc O}\left(h^2\right) ,\quad&
\frac{\partial q^{(3)}}{\partial \dot{q}_1}(0)&=\frac{6}{h^2}\operatorname{Id}+{\mathpzc O}\left(h\right).
\end{alignat*}

Let us compute $D_{13}L_d^e$. Denote by $F$ the right-hand side of \eqref{dLq0}, so
\begin{align*}
\frac{\partial L_d^{e}}{\partial
q_0^i}(q(0),\dot{q}(0),q(h),\dot{q}(h))&=
\left(-\frac{\partial
L}{\partial\dot{q}^i}+\frac{d}{dt}\frac{\partial
L}{\partial\ddot{q}^i}\right)(q(0),\dot{q}(0),\ddot{q}(0), q^{(3)}(0))\\
&=F_i(q(0),\dot{q}(0),\ddot{q}(0), q^{(3)}(0)).
\end{align*}
Recall that $q(0),\dot{q}(0),\ddot{q}(0), q^{(3)}(0)$ are obtained as the initial conditions for the higher-order Euler--Lagrange equations that correspond to the boundary conditions $q(0),\dot{q}(0),q(h),\dot{q}(h)$. We have
\[
F_i=-\frac{\partial
L}{\partial\dot{q}^i}+
\frac{\partial^2 L}{\partial q^j \partial \ddot{q}^i} \dot q^j
+
\frac{\partial^2 L}{\partial \dot q^j \partial \ddot{q}^i} \ddot q^j
+
\frac{\partial^2 L}{\partial \ddot q^j \partial \ddot{q}^i} q^{(3)j}.
\]
Then
\begin{align*}
\frac{\partial^2 L_d^{e}}{\partial
  {q}_1^j\partial {q}_0^i }&=
\frac{\partial F_i}{\partial q^k} \frac{\partial q^k}{\partial q_1^j}
+
\frac{\partial F_i}{\partial \dot q^k} \frac{\partial \dot q^k}{\partial q_1^j}
+
\frac{\partial F_i}{\partial \ddot q^k}\frac{\partial \ddot q^k}{\partial q_1^j}
+
\frac{\partial F_i}{\partial q^{(3)k}}\frac{\partial q^{(3)k}}{\partial q_1^j}=
\frac{\partial F_i}{\partial \ddot q^k}\frac{\partial \ddot q^k}{\partial q_1^j}
+
\frac{\partial F_i}{\partial q^{(3)k}}\frac{\partial q^{(3)k}}{\partial q_1^j}\\
&=\left(-\frac{\partial^2
L}{\partial \ddot q^k\partial\dot{q}^i}+
\frac{\partial^3 L}{\partial \ddot q^k\partial q^j \partial \ddot{q}^i} \dot q^j
+
\frac{\partial^3 L}{\partial \ddot q^k\partial \dot q^j \partial \ddot{q}^i} \ddot q^j
+\frac{\partial^2 L}{\partial \dot q^k \partial \ddot{q}^i}+
\frac{\partial^3 L}{\partial \ddot q^k\partial \ddot q^j \partial \ddot{q}^i} q^{(3)j}\right)\frac{\partial \ddot q^k}{\partial q_1^j}\\
&\quad+
\frac{\partial^2
L}{\partial \ddot q^k\partial\ddot{q}^i}\frac{\partial q^{(3)k}}{\partial q_1^j}\\
&=\left(-\frac{\partial^2
L}{\partial \ddot q^k\partial\dot{q}^i}
+\frac{\partial^2 L}{\partial \dot q^k \partial \ddot{q}^i}+
\frac{d\mathcal{W}_{ik}}{dt}\right) \left(\frac{6}{h^2}\delta_j^k+\mathpzc{O}(h^2)\right)+\frac{\partial^2
L}{\partial \ddot q^k\partial\ddot{q}^i} \left( -\frac{12}{h^3}\delta_j^k+\mathpzc{O}(h)\right).
\end{align*}
In the expression above, the derivatives are evaluated at the arguments corresponding to time $0$ for each function. It is important to note that the first factor involves $\ddot q(0)$ and $q^{(3)}(0)$, which can blow up for $h \to 0$, even in the simple case of cubic splines. However, for $L$ of the type \eqref{special_Lagrangian} we have
\begin{equation*}
\frac{\partial^2
L}{\partial \ddot q^k\partial\dot{q}^i}=\frac{\partial f_k}{\partial \dot q^i},\qquad
\frac{\partial^2
L}{\partial \dot q^k\partial\ddot{q}^i}=\frac{\partial f_i}{\partial \dot q^k},\qquad
\frac{d\mathcal{W}_{ik}}{dt}= \frac{d}{dt}\frac{\partial^2
L}{\partial \ddot q^k\partial\ddot{q}^i}=\frac{d}{dt}g_{ik}=\frac{\partial g_{ik}}{\partial q^l} \dot q^l.
\end{equation*}
These expressions do not contain $\ddot q$ or $q^{(3)}$, so they are $\mathpzc{O}(1)$ for $h \to 0$. Therefore,
\begin{equation*}
D_{13}L_d^e(q(0), \dot{q}(0), q(h), \dot{q}(h))=\frac{\partial^2 L_d^{e}}{\partial
  {q}_0\partial {q}_1 }(q(0), \dot{q}(0), q(h), \dot{q}(h))=-\frac{12}{h^3}\mathcal{W}+\mathpzc{O} \left(\frac{1}{h^2}\right).
\end{equation*}

The remaining derivatives in $\mathcal{W}_d$ can be computed without using the special form \eqref{special_Lagrangian} of the Lagrangian.
\begin{align*}
D_{14}L_d^e(q(0), \dot{q}(0), q(h), \dot{q}(h))&=\frac{\partial^2 L_d^{e}}{\partial
{q}_0\partial \dot{q}_1 }(q(0), \dot{q}(0), q(h), \dot{q}(h))=\frac{6}{h^2}\mathcal {W}+{\mathpzc O}\left(\frac{1}{h}\right)\\
D_{23}L_d^e(q(0), \dot{q}(0), q(h), \dot{q}(h))&=\frac{\partial^2 L_d^{e}}{\partial
\dot{q}_0\partial {q}_1 }(q(0), \dot{q}(0), q(h), \dot{q}(h))=\frac{6}{h^2}\mathcal {W}+{\mathpzc O}\left(\frac{1}{h}\right)\\
D_{24}L_d^e(q(0), \dot{q}(0), q(h), \dot{q}(h))&=\frac{\partial^2 L_d^{e}}{\partial
\dot{q}_0\partial \dot{q}_1 }(q(0), \dot{q}(0), q(h), \dot{q}(h))=-\frac{2}{h}\mathcal {W}+{\mathpzc O}\left(1\right).
\end{align*}

Seeing $\mathcal{W}_d$ as a block matrix, a well-known result from linear algebra leads us to
\[
\det \mathcal{W}_d= \left(-\frac{12}{h^4}\right)^{\dim Q}\det \mathcal{W}^2+{\mathpzc O} \left( \frac{1}{h^{4\dim Q-1}}\right).
\]
That is, for small enough $h$, if $L$ is regular then $L_{d}^{e}$ is regular.
\end{proof}

In what follows we denote $(TQ\times TQ)_2$ the subset of $(TQ\times
TQ)\times (TQ\times TQ)$ given by \[(TQ\times
TQ)_2:=\{(q_0,\dot{q}_0,q_1,\dot{q}_1,\tilde{q}_1,\dot{\tilde{q}}_1,q_2,\dot{q}_2)\mid
\bar{\pi}_2(q_0,\dot{q}_0,q_1,\dot{q}_1)=\bar{\pi}_1(\tilde{q}_1,\dot{\tilde{q}}_1,q_2,\dot{q}_2)\}.\]

If $L\colon T^{(2)}Q\to\R$ is a regular Lagrangian then the Euler--Lagrange
equations for $L$ gives rise a system of explicit $4$-order
differential equations \[q^{(4)}=\Psi(q,\dot{q},\ddot{q},q^{(3)}).\]
Therefore, for $h$ given, it is possible to derive the following
application (see \cite{higherorderbook})
\[\Psi_{L}^{h}\colon T^{(3)}Q\to T^{(3)}Q\] which maps $(q(0),\dot{q}(0),\ddot{q}(0),q^{(3)}(0))\in T^{(3)}Q$
into $(q(h),\dot{q}(h),\ddot{q}(h),q^{(3)}(h))\in T^{(3)}Q$.
Therefore, from Theorem \ref{Theo1} we deduce the commutativity the diagram in Figure \ref{diagram:LEH}.

\begin{figure}[h!]
\begin{center}
\begin{tikzpicture}[scale=0.8, every node/.style={scale=0.9}]
\path (2.0,4) node(a) {$(q(0),\dot{q}(0),q(h),\dot{q}(h))$};
\path (-1,0) node(c) {$(q(0),\dot{q}(0),-D_1L_d^{e},-D_2L_d^{e})$};
\path (5,0) node(d) {$(q(h),\dot{q}(h),D_3L_d^{e},D_4L_d^{e})$};
\path (-1,-4) node(cc) {$(q(0),\dot{q}(0),\ddot{q}(0),q^{(3)}(0))$};
\path (5,-4) node(dd) {$(q(h),\dot{q}(h),\ddot{q}(h),q^{(3)}(h))$};

\path (0.65,2.3) node[anchor=east] (g) {$\mathbb{F}^{-}L_d^{e}$};
\path (3.45,2.3) node[anchor=west] (h) {$\mathbb{F}^{+}L_d^{e}$};
\path (-1,-2) node[anchor=east] (k) {$\mathbb{F}L$}; \path (5,-2)
node[anchor=east] (l) {$\mathbb{F}L$}; \path (2.0,-4)
node[anchor=north] (k) {$\Psi_L^h$};

\draw[thick,black,|->] (a)  -- (c);
 \draw[thick,black,|->] (a) --
(d);

\draw[thick,black,|->] (cc)  -- (c); \draw[thick,black,|->] (cc)  --
(dd); \draw[thick,black,|->] (dd)  -- (d);

\end{tikzpicture}
\caption[Correspondence between the exact discrete Legendre
transformations and the continuous Hamiltonian flow]{Correspondence
between the discrete Legendre transforms and the continuous
Hamiltonian flow.}\label{diagram:LEH}
\end{center}
\end{figure}
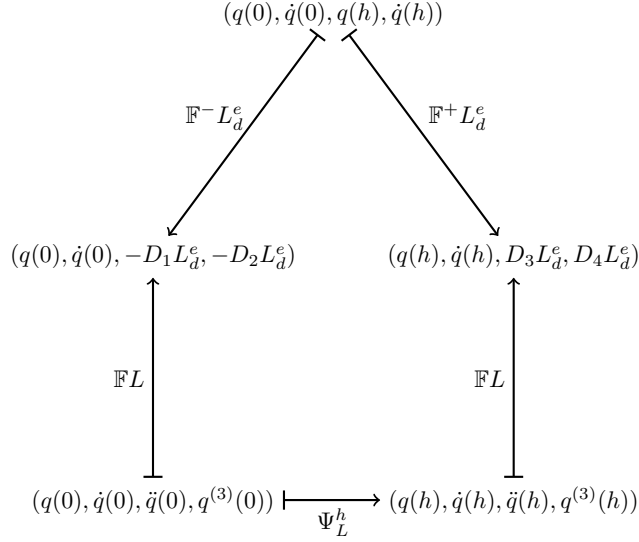

\begin{definition}
The \textit{discrete Hamiltonian flow} is defined by
$\widetilde{F}_{L_d}\colon T^{*}TQ\to T^{*}TQ$ as
\begin{equation}\label{DHF}
\widetilde{F}_{L_d}=\F^{-}L_d\circ F_{L_d}\circ
(\F^{-}L_d)^{-1}.
\end{equation}
Alternatively, it can also be defined as
$\widetilde{F}_{L_d}=\F^{+}L_d\circ F_{L_d}\circ (\F^{+}L_d)^{-1}$.
\end{definition}

\begin{theorem}\label{Theo2}
The diagram in Figure \ref{fig:discretemaps} is commutative.
\begin{figure}[htbp]
\begin{center}
\begin{tikzpicture}[scale=0.8, every node/.style={scale=0.9}]

\path (2.5,4) node(a) {$(q_0,\dot{q}_0,q_1,\dot{q}_1)$}; \path
(7.5,4) node(b) {$(q_1,\dot{q}_1,q_2,\dot{q}_2)$}; \path (0,0)
node(c) {$(q_0,\dot{q}_0,-D_1L_d,-D_2L_d)$}; \path (5,0) node(d)
{$(q_1,\dot{q}_1,D_3L_d,D_4L_d)$}; \path (10,0) node(e)
{$(q_2,\dot{q}_2,-D_1L_d,-D_2L_d)$};

\path (5.0,4.0) node[anchor=south] (f) {$F_{L_d}$}; \path (1.25,2)
node[anchor=east] (g) {$\mathbb{F}^{-}L_d$}; \path (3.9,2)
node[anchor=west] (h) {$\mathbb{F}^{+}L_d$}; \path (6.3,2)
node[anchor=west] (i) {$\mathbb{F}^{-}L_d$}; \path (8.9,2)
node[anchor=west] (j) {$\mathbb{F}^{+}L_d$}; \path (2.5,0)
node[anchor=north] (k) {$\tilde{F}_{L_d}$}; \path (7.5,0)
node[anchor=north] (l) {$\tilde{F}_{L_d}$}; \draw[thick,black,|->]
(a) -- (b); \draw[thick,black,|->] (a)  -- (c);
\draw[thick,black,|->] (a)  -- (d); \draw[thick,black,|->] (b) --
(d); \draw[thick,black,|->] (b)  -- (e); \draw[thick,black,|->] (c)
-- (d); \draw[thick,black,|->] (d)  -- (e);

\end{tikzpicture}
\caption[Correspondence between the discrete Lagrangian and the
 discrete Hamiltonian map]{Correspondence between the discrete
Lagrangian and the discrete Hamiltonian maps.}\label{diagram:LH}
\label{fig:discretemaps}
\end{center}
\end{figure}
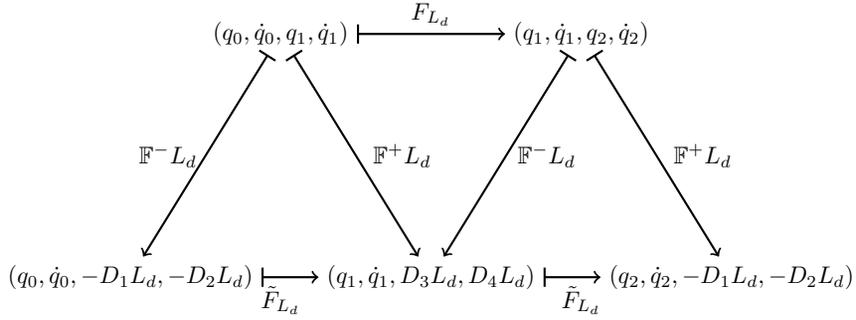

\end{theorem}
\begin{proof}
The central triangle is \eqref{relationF}.
The
parallelogram on the left-hand side is commutative by \eqref{DHF}, so the triangle on the left is commutative. The triangle on the right is the same as the triangle on the left, with shifted indices. Then parallelogram on the right-hand side is commutative, which gives the equivalence stated in the definition of the discrete Hamiltonian flow.
\end{proof}

\begin{corollary}\label{corollaryrelations} The following definitions of the discrete
Hamiltonian map are equivalent
\begin{align*}
\widetilde{F}_{L_d}&=\F^{+}L_d\circ F_{L_d}\circ(\F^{+}L_d)^{-1},\\
\widetilde{F}_{L_d}&=\F^{-}L_d\circ F_{L_d}\circ(\F^{-}L_d)^{-1},\\
\widetilde{F}_{L_d}&=\F^{+}L_d\circ(\F^{-}L_d)^{-1},
\end{align*} and have the coordinate expression $\widetilde{F}_{L_{d}}\colon (q_0,\dot{q}_{0},p_0,\tilde{p}_0)\mapsto (q_1,\dot{q}_1,p_1,\tilde{p}_1)$,
where we use the notation \begin{align*} p_0&=-D_1L_d(q_0,\dot{q}_0,q_1,\dot{q}_1),\\
\tilde{p}_0&=-D_2L_d(q_0,\dot{q}_0,q_1,\dot{q}_1),\\
p_1&=D_3L_d(q_0,\dot{q}_0,q_1,\dot{q}_1),\\
\tilde{p}_1&=D_4L_d(q_0,\dot{q}_0,q_1,\dot{q}_1).
\end{align*}
\end{corollary}

 Combining Theorem \eqref{Theo1} with the diagram in Figure
\ref{fig:discretemaps} gives the commutative diagram shown in Figure
\ref{diagram:LEH2} for the exact discrete Lagrangian.

\begin{figure}[h!]
\begin{center}
\begin{tikzpicture}[scale=0.8, every node/.style={scale=0.9}]
\path (2.0,4) node(a) {$(q_0,\dot{q}_0,q_1,\dot{q}_1)$};
\path (8,4) node(b)  {$(q_1,\dot{q}_1,q_2,\dot{q}_2)$};
\path (-1,0) node(c) {$(q_0,\dot{q}_0,p_0,\tilde{p}_0)$};
\path (5,0) node(d) {$(q_1,\dot{q}_1,p_1,\tilde{p}_1)$};
\path (11,0) node(e) {$(q_2,\dot{q}_2,p_2,\tilde{p}_2)$};
\path (-1,-4) node(cc) {$(q(0),\dot{q}(0),\ddot{q}(0),q^{(3)}(0))$};
\path (5,-4) node(dd) {$(q(h),\dot{q}(h),\ddot{q}(h),q^{(3)}(h))$};
\path (11,-4) node(ee) {$(q(2h),\dot{q}(2h),\ddot{q}(2h),q^{(3)}(2h))$};

\path (5.0,4.0) node[anchor=south] (f) {$F_{L_d^e}$}; \path
(0.65,2.3) node[anchor=east] (g) {$\mathbb{F}^{-}L_d^{e}$}; \path
(3.45,2.3) node[anchor=west] (h) {$\mathbb{F}^{+}L_d^{e}$}; \path
(6.5,1.8) node[anchor=west] (i) {$\mathbb{F}^{-}L_d^{e}$}; \path
(9.7,1.8) node[anchor=west] (j) {$\mathbb{F}^{+}L_d^{e}$}; \path
(2.0,0) node[anchor=north] (k) {$\tilde{F}_{L_d^e} = F_H^h$}; \path
(8,0) node[anchor=north] (l) {$\tilde{F}_{L_d^e} = F_H^h$};

\path (-1,-2) node[anchor=east] (k) {$\mathbb{F}L$}; \path (5,-2)
node[anchor=east] (l) {$\mathbb{F}L$}; \path (11,-2)
node[anchor=east] (l) {$\mathbb{F}L$};

\path (2.0,-4) node[anchor=north] (k) {$F_L^h$}; \path (8,-4)
node[anchor=north] (l) {$F_L^h$}; \draw[thick,black,|->] (a) -- (b);
\draw[thick,black,|->] (a)  -- (c); \draw[thick,black,|->] (a) --
(d); \draw[thick,black,|->] (b) -- (d); \draw[thick,black,|->] (b)
-- (e); \draw[thick,black,|->] (c) -- (d); \draw[thick,black,|->]
(d)  -- (e);

\draw[thick,black,|->] (cc)  -- (c); \draw[thick,black,|->] (cc)  --
(dd); \draw[thick,black,|->] (dd)  -- (d); \draw[thick,black,|->]
(dd)  -- (ee); \draw[thick,black,|->] (ee)  -- (e);

\end{tikzpicture}
\caption[Correspondence between the exact discrete Lagrangian and
the continuous Hamiltonian flow]{Correspondence between the exact
discrete Lagrangian and the continuous Hamiltonian
flow.}\label{diagram:LEH2}
\end{center}
\end{figure}
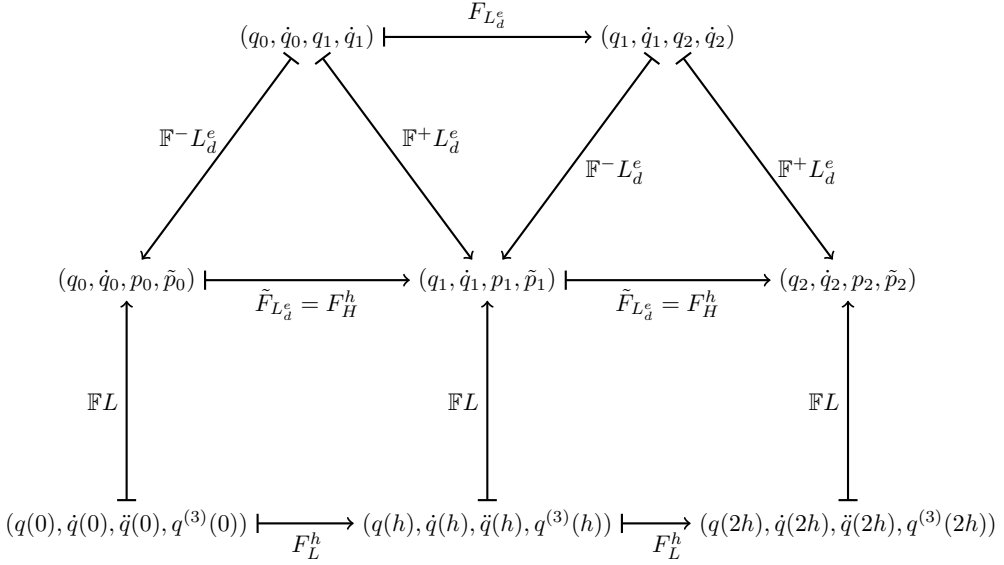

Here, $F_{H}^{h}$ denotes the flow of the Hamiltonian vector field
$X_{H}$ associated with the Hamiltonian $H\colon T^{*}TQ\to\R$ given by
$H=E_{L}\circ(\F L)^{-1}$ where $E_{L}\colon T^{(3)}Q\to\R$ denotes the
energy function associated to $L$ (see \cite{LR1}).

\begin{theorem}\label{Theo3}
Under these conditions we have that $F_H^h = \tilde{F}_{L_d^e}$.
\end{theorem}

\begin{example}\textbf{Cubic splines (cont.)}
Recall that in this example $Q=\R^n$ and $L= \frac{1}{2}\ddot q^2$. Since the exact solutions for the second-order Euler--Lagrange equation for $L$ can be found explicitly, it is easy to show that the discrete exact Lagrangian is
\[
L_d^e(q_0,v_0,q_1,v_1)=\frac{6}{h^3}(q_0-q_1)^2+\frac{6}{h^2}(q_0-q_1)(v_0+v_1)+\frac{2}{h}(v_0^2+v_0v_1+v_1^2).
\]
{}From the corresponding discrete second-order Euler--Lagrange equation, the evolution is
\begin{align*}
q_{k+1}&=5q_{k-1}-4q_k+2h(v_{k-1}+2v_k),\\
v_{k+1}&=v_{k-1}+\frac{2}{h}(q_{k-1}-2q_k+q_{k+1}).
\end{align*}
It is interesting to note that both this exact method and method
\eqref{metodosplineTaylor} preserve the quantity
\[
\varphi(q_k,v_k,q_{k+1},v_{k+1})=\frac{q_{k+1}-q_k}{h}-\frac{v_k+v_{k+1}}{2}.
\]

\begin{figure}[h]
\includegraphics[trim = 20mm 16mm 20mm 20mm, clip, scale=.5]{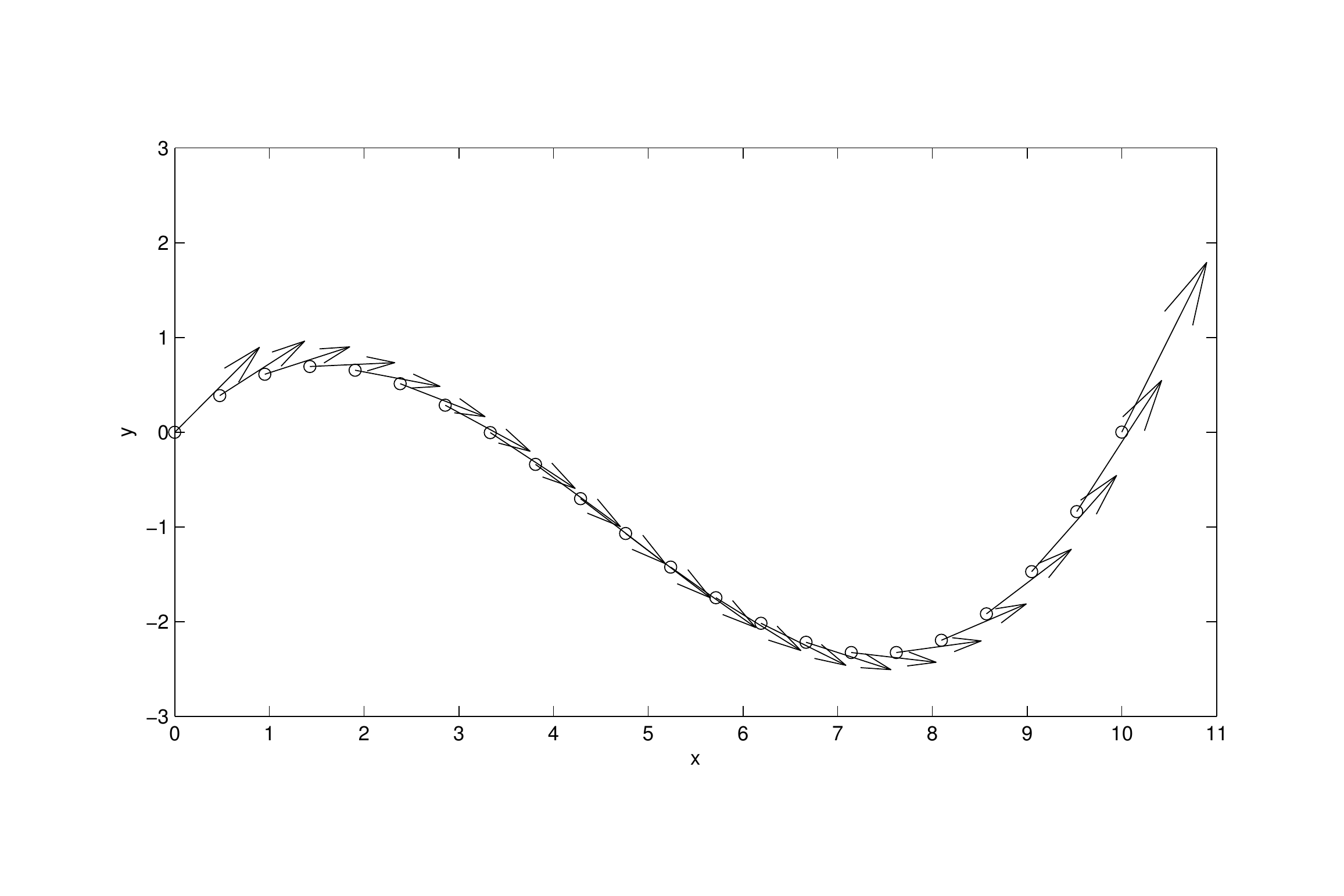}\hspace{4mm}
\includegraphics[trim = 73mm 0mm 72mm 0mm, clip, scale=.35]{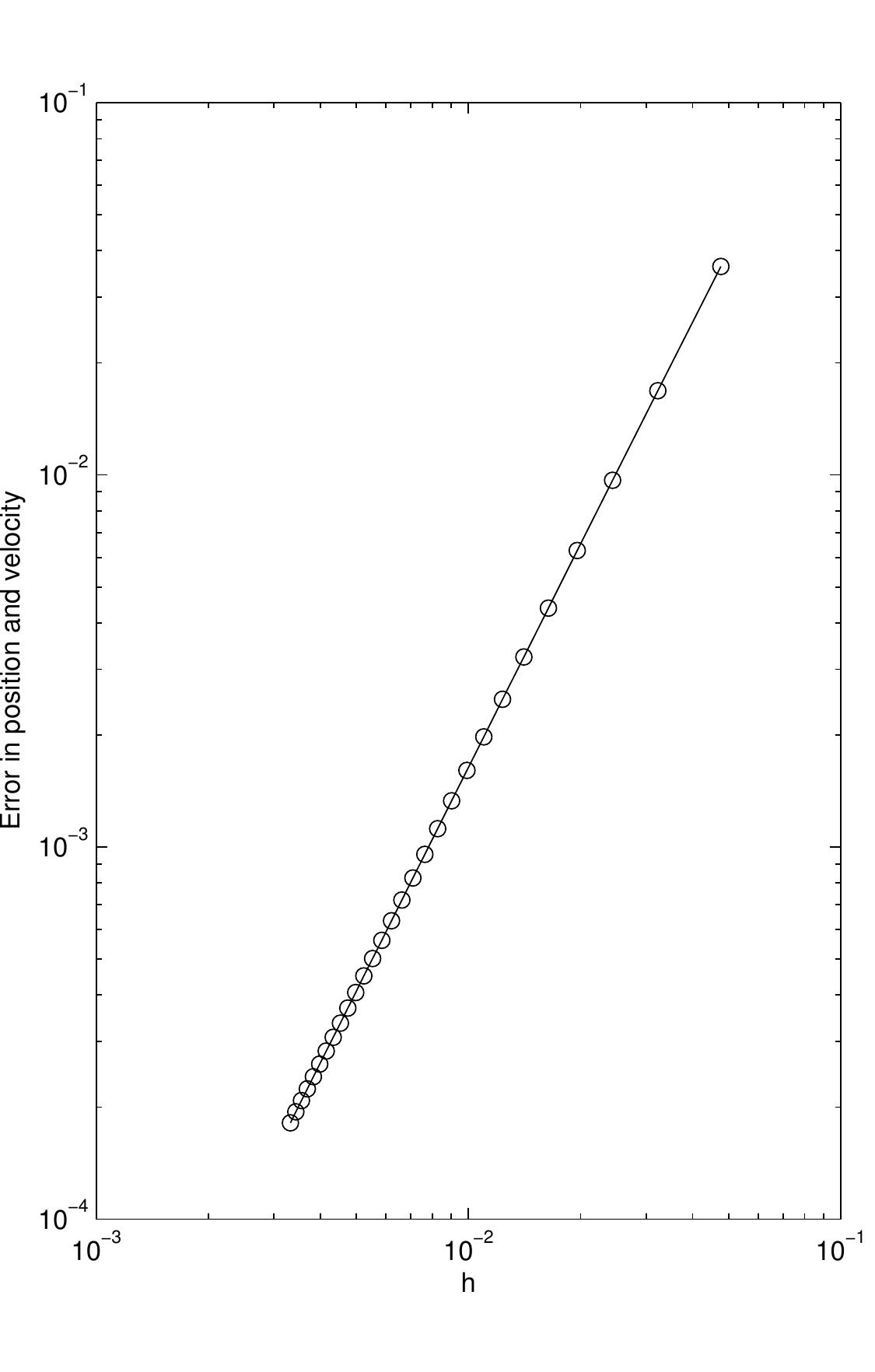}
\caption{Left: simulation of the method \eqref{metodosplineTaylor} with $q_0=(0,0)$ $v_0=(10,10)$, $q_N=(10,0)$, $v_N=(10,20)$, $N=21$, depicting the computed points and velocities in the $xy$-plane (velocities are scaled). Right: Error in position and velocity for different values of $h$.}\label{fig:simulationspline}
\end{figure}
\end{example}

\subsection{Variational error analysis}

Now we rewrite the result of Patrick \cite{patrick} and Marsden and West
\cite{mawest} for the particular case of a Lagrangian
$L_d\colon TQ\times TQ\to\R$.

\begin{definition}
Let $L_{d}\colon TQ\times TQ\to\R$ be a discrete Lagrangian. We say that
$L_{d}$ is a discretization of order $r$ if there exist an
open subset $U_{1}\subset T^{(2)}Q$ with compact closure and
constants $C_1>0$, $h_1>0$ so that

\begin{equation*}
|L_{d}(q(0),\dot{q}(0),q(h),\dot{q}(h),h)-L_{d}^{e}(q(0),\dot{q}(0),q(h),\dot{q}(h),h)|\leq C_{1}h^{r+1}
\end{equation*} for all solutions $q(t)$ of the second-order Euler--Lagrange equations with initial conditions $(q_0,\dot{q}_0,\allowbreak\ddot{q}_0)\in U_1$ and for all $h\leq h_1$.
\end{definition}

Following \cite{mawest, patrick-cuell}, we have the next result about the order of our variational integrator.

\begin{theorem}
If $\widetilde{F}_{L_d}$ is the evolution map of an order $r$
discretization $L_d\colon TQ\times TQ\to\R$ of the exact discrete
Lagrangian $L_d^{e}\colon TQ\times TQ\to\R$, then
\[\widetilde{F}_{L_d}=\widetilde{F}_{L_{d}^{e}}+\mathcal{O}(h^{r+1}).\]
In other words, $\widetilde{F}_{L_d}$ gives an integrator of order
$r$ for $\widetilde{F}_{L_{d}^{e}}=F_{H}^{h}$.
\end{theorem}

Note that given a discrete Lagrangian $L_{d}\colon TQ\times TQ\to\R$  its
order can be calculated by expanding the expressions for
$L_d(q(0),\dot{q}(0),q(h),\dot{q}(h), h)$ in a Taylor series in $h$
and comparing this to the same expansions for the exact Lagrangian.
If the series agree up to $r$ terms, then the discrete Lagrangian is
of order $r$.

\section{Application to optimal control of mechanical systems}\label{section5}

In this section we will study how to apply our variational
integrator to optimal control problems. We will study optimal
control problems for fully actuated mechanical systems and we will
show how our methods can be applied to the optimal control of a
robotic leg.

In the following we will assume that all the control systems are
controllable, that is, for any two points $q_0$ and $q_f$ in the
configuration space $Q$, there exists an admissible control $u(t)$
defined on some interval $[0,T]$ such that the system with initial
condition $q_0$ reaches the point $q_f$ at time $T$ (see \cite{Blo}
and \cite{bullolewis} for example).

\subsection{Optimal control of fully actuated systems.}

Let $L\colon TQ\to\R$ be a regular Lagrangian and take local
coordinates $(q^{A})$ on $Q$ where $1\leq A\leq n$.  For this
Lagrangian the \textit{controlled  Euler--Lagrange equations} are
\begin{equation}\label{E-L}
\frac{d}{dt}\frac{\partial L}{\partial \dot{q}^{A}}-\frac{\partial L}{\partial q^{A}}=u_{A},
\end{equation} where $u=(u_{A})\in U\subset\mathbb{R}^{n}$ is an open subset of $\R^n$, the
set of control parameters.

The optimal control problem consists in finding a trajectory of the
 state variables and control inputs $(q^{(A)}(t),u^{A}(t))$
 satisfying \eqref{E-L} given initial and final conditions
 $(q^{A}(t_0),\dot{q}^{A}(t_0))$, $(q^{A}(t_f),\allowbreak\dot{q}^{A}(t_f))$
 respectively, minimizing the cost function
 \[\mathcal{A}=\int_{t_0}^{t_f} C(q^{A},\dot{q}^{A},u_{A})dt,\] where $C\colon TQ\times U\to\R$.

 From \eqref{E-L} we can rewrite the cost function as a second-order
 Lagrangian $\widetilde{L}\colon T^{(2)}Q\to\R$ given by \[\widetilde{L}(q^{A},\dot{q}^{A},\ddot{q}^{A})=C\left(q^{A},\dot{q}^{A},\frac{d}{dt}\frac{\partial L}{\partial\dot{q}^{A}}-\frac{\partial L}{\partial q^{A}}\right)\]
replacing the controls by the Euler--Lagrange equations in the
cost function (see \cite{Blo} for example).

Suppose that $Q=\R^n$. Then we can define a discretization of the Lagrangian $\widetilde{L}\colon T^{(2)}Q\to\R$ by a discrete Lagrangian $\widetilde{L}_d\colon TQ\times
TQ\to\R$,
\begin{align*}
\widetilde{L}_d(q_{k},v_k,q_{k+1},v_{k+1})&=\frac{h}{2}\widetilde{L}\left(\frac{q_k+q_{k+1}}{2},\frac{v_k+v_{k+1}}{2},\frac{2}{h^2}(q_{k+1}-q_{k}-hv_{k})\right)\\
&+\frac{h}{2}\widetilde{L}\left(\frac{q_k+q_{k+1}}{2},\frac{v_k+v_{k+1}}{2},\frac{2}{h^2}(q_{k}-q_{k+1}+hv_{k+1})\right).
\end{align*}
In the first term, we have computed an approximate value of the acceleration $a_k$ by using the Taylor expansion $q_{k+1}\approx q_k+hv_k+\frac{h^2}{2}a_k$. For the second term, we have approximated $a_{k+1}$ using $q_{k}\approx q_{k+1}-hv_{k+1}+\frac{h^2}{2} a_{k+1}$, as in Example \ref{example_cubic_splines}.

Other natural possibilities for $\widetilde{L}_d$ are, for instance,
\begin{align*}
\widetilde{L}_d(q_k,v_k,q_{k+1},v_{k+1})&=hL\left(\frac{q_k+q_{k+1}}{2},\frac{q_{k+1}-q_k}{h},\frac{v_{k+1}-v_k}{h}\right)\intertext{or}
\widetilde{L}_d(q_k,v_k,q_{k+1},v_{k+1})&=\frac{1}{2}L\left(q_k,v_k,\frac{v_{k+1}-v_{k}}{h}\right)+\frac{1}{2}L\left(q_{k+1},v_{k+1},\frac{v_{k+1}-v_{k}}{h}\right).
\end{align*}

Applying the results given in Section \ref{section2}, we know that
the minimizers of the cost function are obtained by solving the
discrete second-order Euler--Lagrange equations
\begin{align*}
D_1\widetilde{L}_d(q_k,v_k,q_{k+1},v_{k+1})+D_3\widetilde{L}_d(q_{k-1},v_{k-1},q_k,v_k)&=0,\\
D_2\widetilde{L}_d(q_k,v_k,q_{k+1},v_{k+1})+D_4\widetilde{L}_d(q_{k-1},v_{k-1},q_{k},v_k)&=0.
\end{align*}

If the matrix  \[\left(
                 \begin{array}{cc}
                   D_{13}\widetilde{L}_d &  D_{14}\widetilde{L}_d \\
                    D_{23}\widetilde{L}_d &  D_{24}\widetilde{L}_d \\
                 \end{array}
               \right)\]
is regular, then one can define the discrete Lagrangian map to solve the optimal control
problem.

\begin{example}\textbf{Two-link manipulator}

We consider the optimal control of a two-link manipulator which is a
classical example studied in robotics (see for example \cite{Murray}
and \cite{Sina}). The two-link manipulator consists of two coupled
(planar) rigid bodies with mass $m_i$, length $l_i$ and moments of inertia with respect to the joints
$J_i$, with $i = 1, 2$, respectively.

\begin{figure}[h]
\includegraphics[scale=.8]{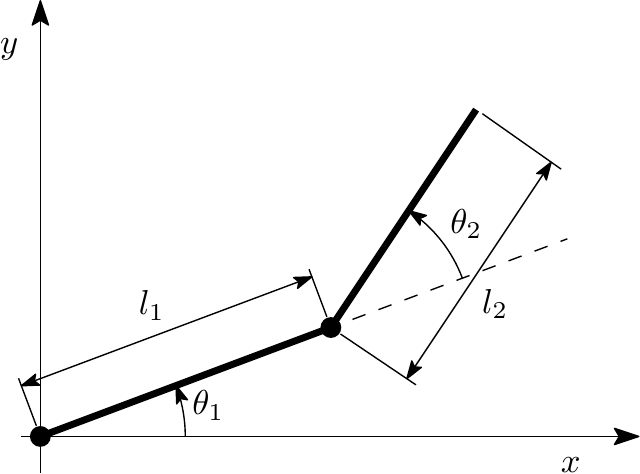}
\caption{Two-link manipulator}\label{fig:two-link}
\end{figure}

Let $\theta_1$ and $\theta_2$
be the configuration angles measured as in Figure~\ref{fig:two-link}. If we assume one end of the first link to be
fixed in an inertial reference frame, the configuration of the
system is locally specified by the coordinates $(\theta_1,
\theta_2)\in\mathbb{S}^{1}\times\mathbb{S}^{1}$. The Lagrangian is given by the kinetic energy of the system minus the
potential energy, that is,
\begin{multline*}
L(q,\dot{q})=\frac{1}{8}(m_{1}+4m_{2})l_{1}^{2}\dot{\theta}_{1}^{2}+\frac{1}{8}m_{2}l_{2}^{2}(\dot{\theta}_{1}+\dot{\theta}_{2})^{2}+\frac{1}{2}m_{2}l_{1}l_{2}\cos(\theta_2)\dot{\theta}_{1}(\dot{\theta}_{1}+\dot{\theta}_2)+\frac{1}{2}J_{1}\dot{\theta}_{1}^{2}\\
\quad+\frac{1}{2}J_{2}(\dot{\theta}_{1}+\dot{\theta}_{2})^{2}+g\left(\frac{1}{2}m_{1}l_{1}\sin\theta_{1}+m_{2}l_{1}\sin\theta_{1}+\frac{1}{2}m_{2}l_{2}(\theta_1+\theta_{2})\right),
\end{multline*} where $g$ is the constant gravitational acceleration.

Control torques $u_{1}$ and $u_{2}$ are applied at the base of
the first link and at the joint between the two links. The equations of
motion of the controlled system are \begin{align*}
u_{1}&=-\sin\theta_2l_1l_2m_2\dot{\theta}_{2}\dot{\theta}_{1}-\frac{1}{2}\sin\theta_{2}\dot{\theta}_{2}^{2}l_1l_2m_{2}+\frac{1}{2}m_2l_2\cos(\theta_1+\theta_2)g\\
&\quad+\left(m_2g\cos\theta_1+\frac{1}{2}g\cos\theta_{1}m_1\right)l_{1}+\left(\frac{1}{4}m_2l_2^{2}+J_2+\frac{1}{2}\cos\theta_2l_1l_2m_2\right)\ddot{\theta}_{2}\\
&\quad+\left(\cos\theta_2l_1l_2m_2+\left(\frac{m_1}{4}+m_2\right)l_{1}^{2}+\frac{m_{2}l_{2}^{2}}{4}+J_{1}+J_2\right)\ddot{\theta}_{1},\\
u_2&=\frac{1}{2}\sin\theta_2l_1l_2m_2\dot{\theta}_{1}^{2}+\left(\frac{1}{4}m_2l_2^{2}+J_2+\frac{1}{2}\cos\theta_{2}l_{1}l_{2}m_{2}\right)\ddot{\theta}_{1}\\
&\quad+\frac{1}{2}m_{2}l_{2}\cos(\theta_1+\theta_2)g+\left(\frac{1}{4}m_{2}l_{2}^{2}+J_{2}\right)\ddot{\theta}_{2}.
\end{align*}

We look for  trajectories $(\theta_{1}(t),\theta_{2}(t), u(t))$ of
the state variables and control inputs for given initial and final
conditions, that is, for given values of $(\theta_{1}(0),\theta_{2}(0), \dot{\theta}_1(0),
\dot\theta_2(0))$ and
 $(\theta_1(T), \theta_2(T), \allowbreak \dot\theta_1(T), \dot\theta_2(T))$, and minimizing the cost functional
\[\mathcal{A} = \frac{1}{2}\int_{0}^{T} (u_{1}^{2}+u_{2}^{2})\,dt.\]

We construct the discrete Lagrangian
$\widetilde{L}_d\colon T(\mathbb{S}^{1}\times\mathbb{S}^{1})\times
T(\mathbb{S}^{1}\times\mathbb{S}^{1})\to\R$, discretizing the
Lagrangian $\displaystyle{\widetilde{L}\colon T^{(2)}(\mathbb{S}^{1}\times\mathbb{S}^{1})\to\R}$ given by
\begin{align*}
\widetilde{L}(\theta_1,\theta_2,\dot{\theta}_1,\dot{\theta}_2,\ddot{\theta}_1,\ddot{\theta}_2)&=\frac{1}{2}\left[\frac{1}{2}\sin\theta_2l_1l_2m_2\dot{\theta}_{1}^{2}+\left(\frac{1}{4}m_2l_2^{2}+J_2+\frac{1}{2}\cos\theta_{2}l_{1}l_{2}m_{2}\right)\ddot{\theta}_{1}\right.\\
&\quad+\left.\frac{1}{2}m_{2}l_{2}\cos(\theta_1+\theta_2)g+\left(\frac{1}{4}m_{2}l_{2}^{2}+J_{2}\right)\ddot{\theta}_{2}\right]^{2}\\
&\quad+\frac{1}{2}\left[\frac{1}{2}\sin\theta_2l_1l_2m_2\dot{\theta}_{1}^{2}+\left(\frac{1}{4}m_2l_2^{2}+J_2+\frac{1}{2}\cos\theta_{2}l_{1}l_{2}m_{2}\right)\ddot{\theta}_{1}
\right.\\
&\quad+\left.\frac{1}{2}m_{2}l_{2}\cos(\theta_1+\theta_2)g+\left(\frac{1}{4}m_{2}l_{2}^{2}+J_{2}\right)\ddot{\theta}_{2}
\right]^{2}
\end{align*}
taking the same discretization as in equation \eqref{discrete
Taylor} to approximate the acceleration and taking midpoint averages
to approximate the position and velocity.

Figures \ref{fig:manipulator-position} and \ref{fig:manipulator-3d} show the results from a numerical simulation of the method, taking the system from the stable mechanical equilibrium $(\theta_{1}(0),\theta_{2}(0), \dot{\theta}_1(0),
\dot\theta_2(0))=(-\pi/2,0,0,0)$ to the unstable equilibrium  $(\theta_{1}(T),\theta_{2}(T), \dot{\theta}_1(T),
\dot\theta_2(T))=(\pi/2,0,0,0)$. We have used $T=10$, $N=1000$, $m_1=0.375$, $m_2=0.25$, $l_1=1.5$, $l_2=1$, $J_1=\frac{m_1l_1^2}{3}$,  $J_2=\frac{m_2l_2^2}{3}$, and $g=9.8$. In addition, the reader can find a video of the simulation in \url{www.youtube.com/watch?v=ZUUH0596a30}.
The algorithm generates a sequence of velocities as well as positions, but we represent only the positions in the figures.
\begin{figure}[h]
\includegraphics[trim = 7mm 0mm 0mm 0mm, clip, scale=.7]{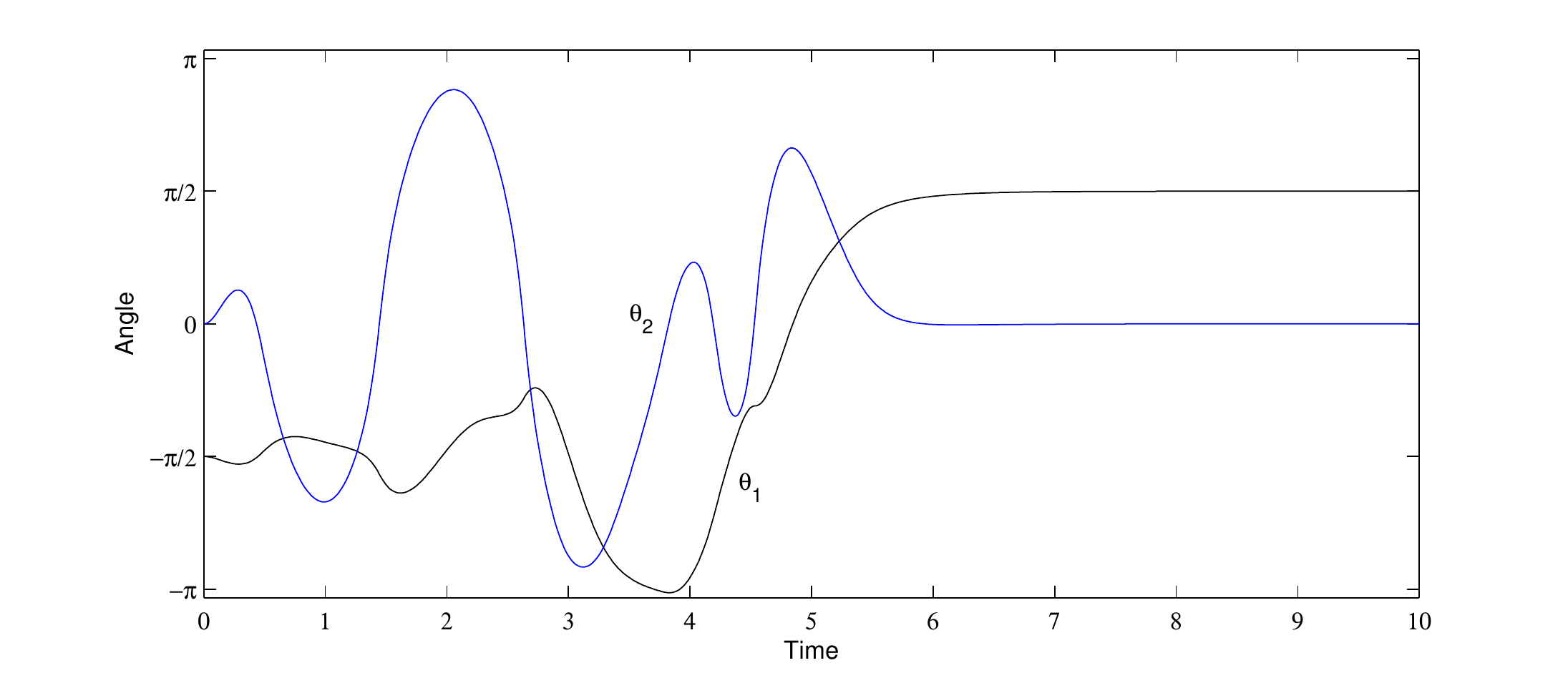}
\caption{Angles $\theta_1$ and $\theta_2$ for the optimal control of the two-link manipulator. Initially, the two links point downwards; at $T=10$ they point upwards.} \label{fig:manipulator-position}
\end{figure}
\begin{figure}
\includegraphics[trim = 85mm 10mm 85mm 10mm, clip, scale=.6]{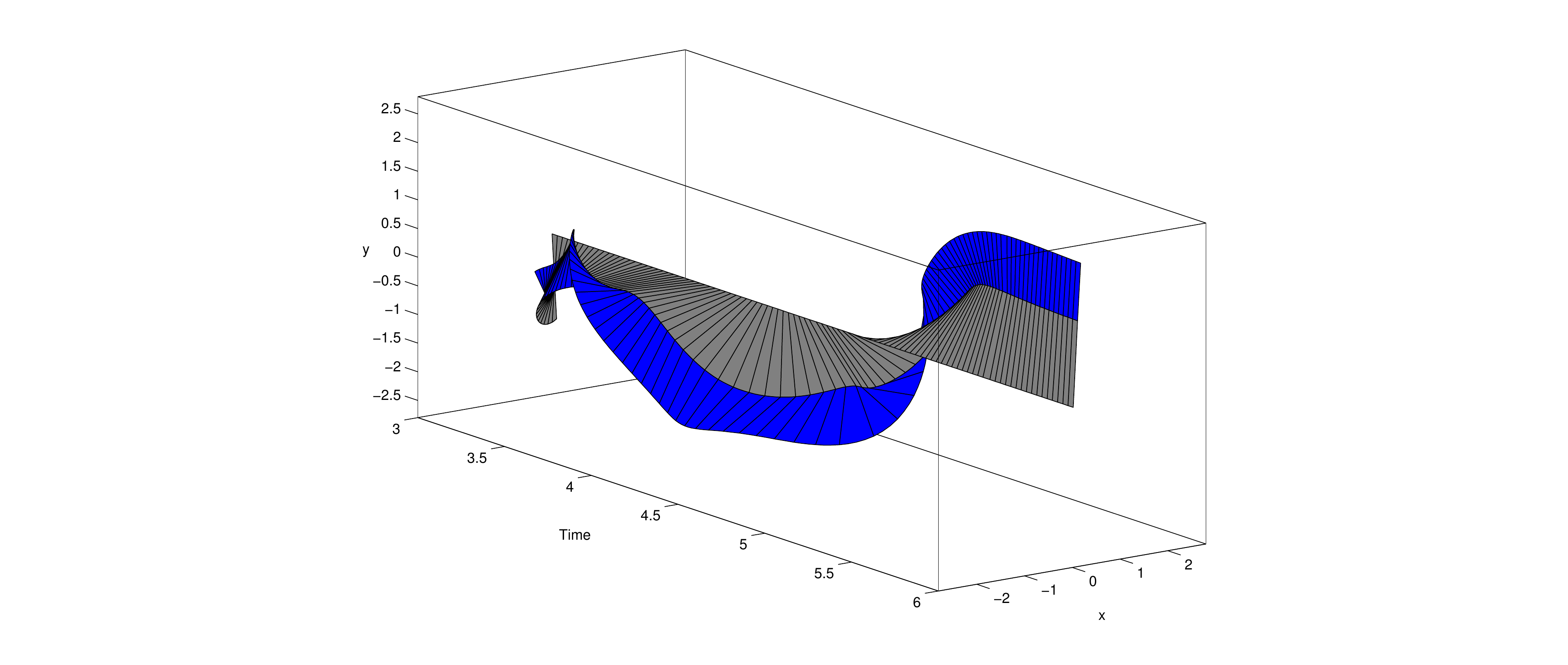}
\caption{Evolution of the actual
 position of the two-link manipulator (detail for $t\in[3,6]$). Sections of this surface with the vertical plane $t=t_0$ show the two links as they are positioned at time $t_0$.}\label{fig:manipulator-3d} 
\end{figure}

We have also considered a different setting where the angle $\theta_2$ is restricted to move between 0 and 170 degrees, inspired by an elbow joint. This range of motion is enforced by adding a continuous, piecewise linear function $V(\theta_2)$ to the cost function, with slope $-1000$ for $\theta_2<0^\circ$, $0$ for $0^\circ<\theta_2<170^\circ$, and $1000$ for $\theta_2>170^\circ$. We  simulated the optimal trajectory with the same endpoint conditions and physical parameters as above, with $N=200$. A video of the resulting motion can be found in \url{www.youtube.com/watch?v=OxOFHdT7emQ}.

\end{example}

\section*{Conclusions and future research}
In this paper we design variational integrators for
higher-order variational systems and their application to optimal
control problems. The general idea for those variational integrators
is to directly discretize Hamilton's principle rather than the
equations of motion in a way that preserves the original system
invariants, notably the symplectic form and, via a discrete version
of Noether's theorem, the momentum map.

We show that a regular higher-order Lagrangian
system has a unique solution for given nearby endpoint conditions
using a direct variational proof of existence and uniqueness for the
local boundary value problem using a regularization procedure
assuming only $C^k$ differentiability (instead of $C^{2k}$ as in
standard ODE theory).

We have seen that taking a discrete Lagrangian
function $L_d\colon  T^{(k-1)}Q\times T^{(k-1)}Q \to{\mathbb R}$ we obtain
the appropriate approximation of the action $ \int^h_0 L(q, \dot{q},
\ldots, q^{(k)})\, dt$. Moreover, we derive a  particular choice of
discrete Lagrangian which gives an exact correspondence between
discrete and continuous systems, the exact discrete Lagrangian. We
show that if the original Lagrangian is regular then it is also the
exact discrete Lagrangian and how is the relation between the
discrete Legendre transformations with the continuous one. 

As future research, we are interested in the
construction of an exact discrete Lagrangian function for higher-order mechanical systems subject to higher-order constraints. The main point will be to show the existence and uniqueness of solutions for the boundary value problem for higher-order systems subject to higher-order constraints. After it, one could define the exact discrete Lagrangian for constrained systems in a similar way that the ones shown in this work. Since optimal control problems for the class of under actuated mechanical systems can be seen as constrained higher-order variational problems, the extension of the constructions given in this work, can be useful to new developments in the field of geometric integration for optimal control problems.  The case of optimal control of nonholonomic systems will be developed.

\section*{Appendix: a technical result for section 2}\label{appendix}
Let $E$ be the kernel of $g$, where
$g=(g_0,\ldots,g_{k-1})\colon C^{l}([0,1],\R^{n})\to (\R^{n})^{k}$ and $g_{j}[\cdot]=\langle b_{j}^{[k]},\cdot\rrangle$. In the context of section \ref{solregularized}, $E$ is the tangent space of the constraint set defined using the linear constraints $g_{j}$, and $l$ is either $0$ or $k$. 

In this Appendix we show that the orthogonal complement of $E$ is the space
$F$ of $\R^{n}$-valued polynomials of degree at
most $k-1$,
\[F=\hbox{span}_{\R^{n}}(b_{0}^{[k]},\ldots,b_{k-1}^{[k]})=\{c^{j}b_{j}^{[k]}|c^{0},\ldots,c^{k-1}\in\R^{n}\},\]
where $b_j^{[k]}$, $j=0,\ldots,k-1$, is a basis of the space of real-valued polynomials of degree at most $k-1$ consisting of orthonormal polynomials on $[0,1]$.

\begin{lemma}
$F=E^{\perp}$, where the orthogonal complement is taken with
respect to the inner product $\llangle\cdot,\cdot\rrangle$ in $C^{l}([0,1],\R^{n})$.
\end{lemma}

\begin{proof} We will prove that $E$ and $F$ are orthogonal (with zero intersection) and that their sum is the whole space $C^{l}([0,1],\R^{n})$.

Let $e\in E$ and $c^{j}b_{j}^{[k]}\in F$. 
\begin{align*}
\llangle c^{j}b_{j}^{[k]},e\rrangle &=\int_{0}^{1}(c^{j}b_{j}^{[k]}(u))\cdot e(u)\,du=\sum_{i=1}^{n}\int_{0}^{1}c_{i}^{j}b_{j}^{[k]}(u)e_{i}(u)du\\
&=c^{j}\cdot\left(\int_{0}^{1}b_{j}^{[k]}(u)e_{1}(u),\ldots,\int_{0}^{1}b_{j}^{[k]}(u)e_{n}(u)\right)\\
&=c^{j}\cdot\langle b_{j}^{[k]},e\rrangle=c^{j}\cdot g_{j}[e]=0,
\end{align*}
since $e\in E=\Ker g$.

The fact that $E\cap F=\{0\}$ can be obtained either by using that the inner product is nondegenerate or directly as follows. Take $e\in E\cap F$, so $e=c^{j}b_{j}^{[k]}$. For all $j'$, we have $0=g_{j'}[e]=\langle b_{j'}^{[k]},c^{j}b_{j}^{[k]}\rrangle=c^{j'}$, which means that $e=0$. 

Finally, take $e\in C^{l}([0,1],\R^{n})$. Write
\[e=e-\sum_{j=0}^{k-1}\langle b_{j}^{[k]},e\rrangle b_j^{[k]}+\sum_{j=0}^{k-1}\langle b_{j}^{[k]},e\rrangle b_j^{[k]}.\]
The third term is in $F$. The remaining part of the right-hand side is in $E$ since for all $j'$,
\[
\Big\langle b_ {j'},e-\sum_{j=0}^{k-1}\langle b_{j}^{[k]},e\rrangle b_j^{[k]}\Bigrrangle=\langle
b_{j'},e\rrangle-\sum_{j=0}^{k-1}\delta_{j'j}\langle
b_j,e\rrangle =0.\]
Therefore $C^{l}([0,1],\R^{n})=E+F$. From the first part of the proof, we obtain that there is an orthogonal decomposition $C^{l}([0,1],\R^{n})=E\oplus F$.
\end{proof}

\section*{Acknowledgments} We would like to thank Juan Carlos Marrero for useful comments and discussions.
\bibliographystyle{AIMS}
\bibliography{refs}

\end{document}